\documentclass[11pt]{article}
\usepackage[numbers,sort&compress]{natbib}
\usepackage{enumerate}
\usepackage{amscd}
\usepackage{array}
\usepackage{amsmath}
\usepackage{latexsym}
\usepackage{amsfonts}
\usepackage{amssymb}
\usepackage{amsthm}
\usepackage{verbatim}
\usepackage{mathrsfs}
\usepackage{enumerate}
\usepackage{hyperref}

\theoremstyle{plain}
\theoremstyle{definition}\newtheorem{theorem}{Theorem}[section]
\theoremstyle{plain}\newtheorem{lemma}[theorem]{Lemma}
\theoremstyle{plain}\newtheorem{coro}[theorem]{Corollary}
\theoremstyle{plain}\newtheorem{prop}[theorem]{Proposition}
\theoremstyle{remark}\newtheorem{remark}{Remark}[section]
\theoremstyle{definition}
\theoremstyle{plain}
\usepackage{xcolor}

\newcommand{\norm}[1]{\left\|#1\right\|}
\newcommand{\Div}{\mathrm{div}\,}
\newcommand{\B}{\Big}

\newcommand{\be}{\begin{equation}}
\newcommand{\ee}{\end{equation}}
 \newcommand{\ba}{\begin{aligned}}
 \newcommand{\ea}{\end{aligned}}

  \newcommand{\f}{\frac}
    
  \newcommand{\ben}{\begin{enumerate}}
   \newcommand{\een}{\end{enumerate}}

\newcommand{\ti}{\nabla}

\newcommand{\Rmnum}[1]{\expandafter\@slowromancap\romannumeral #1@}

\allowdisplaybreaks

\numberwithin{equation}{section}
\begin{document}%\begin{CJK*}{GBK}{fs}
%%%%%%%%%%%%%%%%%%%%%%%%%%%%%%%%%%%%%%%%%%%%%%%%%%%%%%%%%%%%%%%%%%%%%%%%%%%%%%%%%%%%%%%%%%%%%%%%%%%%
\title{ Gagliardo-Nirenberg inequalities in Lorentz type spaces and energy equality for the Navier-Stokes system}
 \author{ Yanqing Wang\footnote{ Department of Mathematics and Information Science, Zhengzhou University of Light Industry, Zhengzhou, Henan  450002,  P. R. China Email: wangyanqing20056@gmail.com},\;~ Wei Wei\footnote{Corresponding author. Center for Nonlinear Studies, School of Mathematics, Northwest University, Xi'an, Shaanxi 710127,  P. R. China  Email: ww5998198@126.com }
  \;~~and ~~Yulin Ye\footnote{School of Mathematics and Statistics, Henan University, Kaifeng, Henan  475004, P. R. China. Email: ylye@vip.henu.edu.cn }
   }
\date{}
\maketitle
\begin{abstract}
 In this paper, we derive some new Gagliardo-Nirenberg type inequalities in Lorentz type spaces without restrictions on the second index of Lorentz norms, which generalize almost all known corresponding results. Our proof mainly relies on the Bernstein inequalities in Lorentz spaces, the embedding relation among various Lorentz type spaces, and Littlewood-Paley decomposition techniques. In addition, we establish several novel criteria in terms of the velocity or the gradient of the velocity in Lorentz spaces for energy conservation of the 3D Navier-Stokes equations. Particularly, we improve the classical Shinbrot's condition for energy balance to allow both the space-time directions of the velocity to be in Lorentz spaces.

 \end{abstract}
%%%%%%%%%%
\section{Introduction}
\label{intro}
\setcounter{section}{1}\setcounter{equation}{0}
\subsection{Gagliardo-Nirenberg inequality}
An important way to study well-posedness of partial differential equations is
via the research of their weak solutions. It is well known that Gagliardo-Nirenberg inequality is a fundamental tool to improve the regularity of weak solutions, which has gained widespread applications such as in  Hilbert's 19th problem
\cite{[De Giorgi]},   Caffarelli-Kohn-Nirenberg theorem \cite{[CKN]} for the 3D Navier-Stokes equations and the critical quasi-geostrophic equations \cite{[CV]}.
The classical integer version of Gagliardo-Nirenberg inequality is the generalization of Sobolev embedding theorem and was discovered independently by Gagliardo \cite{[Gagliardo]} and Nirenberg \cite{[Nirenberg]} as follows: for all smooth functions $u$ in $\mathbb{R}^{n}$ with compact support, there holds
\be\label{GNI}
\|D^{j}u\|_{L^{p}(\mathbb{R}^{n})}\leq C\|D^{m}u\|_{L^{r}(\mathbb{R}^{n})}^{\theta}\|u\|_{L^{q}(\mathbb{R}^{n})}^{1-\theta},
\ee
where $j,m$ are any integers satisfying  $0\leq j<m$, $1\leq q,r\leq \infty,$ and
\[
\frac{1}{p}-\frac{j}{n}= \theta(\frac{1}{r}-\frac{m}{n})+(1-\theta)\frac{1}{q}
\]
for all $\theta\in[\f{j}{m},\,1]$, unless $1<r<\infty$ and $m-j-\f{n}{r}$ is a nonnegative integer.

Up to now, there have been extensive  investigations
on  Gagliardo-Nirenberg inequalities involving fractional derivatives and various function spaces (see \cite{[BM],[KP],[MRR],[DDN],[HYZ],[MS],[Tartar],[SMR],[Ozawa],[Wadade1],[Wadade2],[Chikami]}). On one hand,
  the most general form of fractional Gagliardo-Nirenberg inequality in Lebesgue spaces was recently presented by  Hajaiej-Molinet-Ozawa-Wang  \cite{[HMOW]} and    by Chikami  \cite{[Chikami]}: for $0 \leq \sigma<s<\infty$ and $1\leq q, r \leq \infty,$ there holds
\be\label{Chikami}
\|\Lambda^{\sigma}u\|_{L^{p}(\mathbb{R}^{n})} \leq C\|u\|_{L^{q}(\mathbb{R}^{n})}^{\theta}\|\Lambda^{s} u\|_{L^{r}(\mathbb{R}^{n})}^{1-\theta}
\ee
with $$
\frac{n}{p}-\sigma=\theta \frac{n}{q}+(1-\theta)\left(\frac{n}{r}-s\right),
$$
where $0 \leq \theta \leq 1-\frac{\sigma}{s}\,\left(\theta \neq 0\right.$ if $\left.s-\sigma \geq \frac{n}{r}\right) $ and $\Lambda^{s}=(-\Delta)^{\f{s}{2}} $ is defined via
$\widehat{\Lambda^{s} f}(\xi)=|\xi|^{s}\hat{f}(\xi).$ For the fractional Gagliardo-Nirenberg inequality in the framework of nonhomogeneous Sobolev spaces, Besov spaces and Fourier-Herz  spaces, we refer the reader to \cite{[BM],[Chikami],[HMOW]} and references therein.  On the other hand, recent progress on Gagliardo-Nirenberg inequality in Lorentz spaces for some special cases was made in \cite{[HYZ],[MRR],[DDN]}.
In particular, Hajaiej-Yu-Zhai \cite{[HYZ]} established the Gagliardo-Nirenberg inequality for Lorentz spaces below: for  $1 \leq p, p_{2}, q, q_{1}, q_{2}<\infty, 0<\alpha<q, 0<s<n$ and $1<p_{1}<n/s$,
\be\label{hyz}
\|u\|_{L^{p, q}\left(\mathbb{R}^{n}\right)} \leq C \left\|\Lambda^{s  } u\right\|_{L^{p_{1}, q_{1}}\left(\mathbb{R}^{n}\right)}^{\frac{\alpha}{q}}\|u\|_{L^{p_{2}, q_{2}}\left(\mathbb{R}^{n}\right)}^{\frac{q-\alpha}{q}},\ee
with
$$
 \frac{\alpha}{q_{1}}+\frac{q-\alpha}{q_{2}}=1~~\text{and}~~
\alpha\left(\frac{1}{p_{1}}-\frac{s}{n}\right)+(q-\alpha) \frac{1}{p_{2}}=\frac{q}{p}.
 $$
 In \cite{[MRR]},
McCormick-Robinson-Rodrigo proved the following inequality
\be\label{mrr}
\|f\|_{L^{p}(\mathbb{R}^{n})} \leq C\|f\|_{L^{q, \infty}(\mathbb{R}^{n})}^{\theta}\|\Lambda^{s}f\|_{L^{2}(\mathbb{R}^{n})}^{1-\theta},
\ee
where
$$
\frac{1}{p}=\frac{\theta}{q}+(1-\theta)\left(\frac{1}{2}-\frac{s}{n}\right),~~ n\left(\frac{1}{2}-\frac{1}{p}\right)<s,~~1\leq q<p<\infty~~\text{and}~~s\geq 0.
$$
Subsequently, this result was improved by
Dao-D\'iaz-Nguyen \cite{[DDN]} as follows: for any $\alpha>0$,
\be\label{DDN}
 \|f\|_{L^{p, \alpha}\left(\mathbb{R}^{n}\right)} \leq C\|f\|_{L^{q, \infty}\left(\mathbb{R}^{n}\right)}^{\theta}\|\Lambda^{s}f\|_{L^{2}
 \left(\mathbb{R}^{n}\right)}^{1-\theta}
 \ee
 with$$
\frac{1}{p}=\frac{\theta}{q}+(1-\theta)\left(\frac{1}{2}-\frac{s}{n}\right),~~ n\left(\frac{1}{2}-\frac{1}{p}\right)<s,~~1\leq q<p<\infty~~\text{and}~~s\geq 0.
 $$
One of main purposes of this paper is to establish the most general form of Gagliardo-Nirenberg
inequality in Lorentz spaces, parallel to the classical version \eqref{GNI} and the fractional case  \eqref{Chikami}.
\begin{theorem}\label{the1.1}
Suppose that $u\in L^{q,\infty}(\mathbb{R}^{n})$ and $\Lambda^{s}u\in L^{r,\infty}(\mathbb{R}^{n})$.
Let $0 \leq \sigma<s<\infty$ and $1< q, r \leq \infty.  $ Then there exists a positive constant $C=C(n,q,p,r, s,\sigma)$ such that
\be\label{GNL}
\|\Lambda^{\sigma}u\|_{L^{p,1}(\mathbb{R}^{n} )} \leq C\|u\|_{L^{q,\infty}(\mathbb{R}^{n} )}^{\theta}\|\Lambda^{s} u\|_{L^{r,\infty}(\mathbb{R}^{n} )}^{1-\theta}
\ee
with
$$
\frac{n}{p}-\sigma=\theta \frac{n}{q}+(1-\theta)\left(\frac{n}{r}-s\right),
$$
where $0 <\theta < 1-\frac{\sigma}{s}\,$ and $\,s-\f{n}{r} \neq \sigma-\frac{n}{p}.$
 \end{theorem}
\begin{remark}
The motivation of this theorem is twofold. On one hand, it is worth noting that Gagliardo-Nirenberg inequalities \eqref{hyz}-\eqref{DDN} in Lorentz spaces mentioned above are only concerned with the subcritical case or the case that $q<p$, and the left-hand side of estimates \eqref{hyz}-\eqref{DDN} is without higher derivatives. On the other hand, since there are some extra restrictions on the second index of Lorentz norms in \eqref{hyz}-\eqref{DDN} and that in the results of \cite{[KP]}, it seems that their versions of Gagliardo-Nirenberg inequalities for Lorentz spaces are not easily applicable in the field of
partial differential equations. We emphasize here that Theorem \ref{the1.1} removes all the aforementioned restrictions and covers most of the possible range of exponents.
\end{remark}
\begin{remark}
It is essential to make an exception that $\,s-\f{n}{r} \neq \sigma-\frac{n}{p}\,$ in this theorem, otherwise the generalized Gagliardo-Nirenberg inequality \eqref{GNL} is invalid for $u(x)=|x|^{s-\f{n}{r}}$. We also remark that the restriction that $q, r >1$ results from an application of Lemma \ref{lem3.1}.
\end{remark}
\begin{remark} A special case of \eqref{GNL} is that
$$\|u\|_{L^{p, p_{1}} (\mathbb{R}^{n} )} \leq C  \|\Lambda^{s   } u \|_{L^{r , \infty}\left(\mathbb{R}^{n}\right)}^{1-\theta}\|u\|_{L^{q, \infty} (\mathbb{R}^{n} )}^{ \theta},~~1<q,r\leq \infty,~1\leq p_{1}\leq \infty,~s>0,$$
where
$$
\frac{n}{p}= (1-\theta)\left(\frac{n}{r}-s\right)+\frac{ n\theta}{q},~~0 <\theta < 1.
$$
This result still extends the aforementioned inequalities \eqref{hyz}-\eqref{DDN} for
Lorentz spaces.
\end{remark}
Inspired by the work \cite{[Chikami]},
we shall divide the proof of Theorem \ref{the1.1} into three cases by the relationship between $s-\sigma$ and $n/r$. Firstly, we focus on the subcritical case $s-\sigma<n/r$.  Rough speaking,  we observe that there exist at least four  kinds of equivalent definitions of Lorentz norms (see  \eqref{inorm} for details). This together with the pointwise interpolation estimate for derivatives \eqref{bcd} given in \cite{[BCD]} enables us to prove that
\be\label{suppre11}
\|\Lambda^{\sigma}u\|_{L^{p,\infty}(\mathbb{R}^{n} )} \leq C\|u\|_{L^{q,\infty}(\mathbb{R}^{n} )}^{1-\frac{\sigma}{s}}\|\Lambda^{s}u\|_{L^{r,\infty}(\mathbb{R}^{n} )}^{\frac{\sigma}{s}},
~~\text{
with}~~
\frac{1}{p}=\left(1-\frac{\sigma}{s}\right) \frac{1}{q}+\frac{\sigma}{s r}.
\ee
Then, making use of the interpolation characteristic \eqref{Interpolation characteristic} and Sobolev inequality in Lorentz spaces, we can obtain the desired estimates in this case. The second case is devoted to dealing with the critical case $s-\sigma=n/r$. To this end, in the spirit of   \cite{[Chikami],[Wadade1],[Wadade2]}, we establish the following estimate in the critical Besov-Lorentz spaces
\be\label{crokey1}
 \ba
\|u\|_{L^{q,l}(\mathbb{R}^{n} )}
\leq C \|u\|_{L^{p,\infty}(\mathbb{R}^{n} )}^{\f{p}{q}}  \| u\|^{1-\f{p}{q}}_{\dot{B}^{\f{n}{r}}_{r,\infty,\infty}},~~\text{with}~~p<q.
\ea
\ee
To achieve this, various Bernstein inequalities  \eqref{bern1}-\eqref{bern4} in  Lorentz spaces are derived. The new Bernstein inequality \eqref{bern4}  allows us to get the fact that
\be\label{keyinclue1}
 \|f\|_{ \dot{B}^{s}_{p,\infty,\infty }}\leq C\|\Lambda^{s}f\|_{L^{p,\infty}(\mathbb{R}^{n})}.
\ee
Combining \eqref{suppre11}, \eqref{crokey1} and \eqref{keyinclue1},  we may prove the
critical case.  For the third case, we proceed with the case $s-\sigma> n/r$ by setting up the following key estimate
$$
\|u\|_{L^{\infty}(\mathbb{R}^{n} )} \leq C\|u\|_{L^{p,\infty}(\mathbb{R}^{n})}^{\theta}\|u\|_{\dot{B}^{s}_{r,\infty,\infty}}^{1-\theta},~~\text{with}~~0=\theta \frac{n}{p}+(1-\theta)\left(\frac{n}{r}-s\right),~~0<\theta\leq 1.
$$
By a similar argument used in the previous two cases, this yields the generalized Gagliardo-Nirenberg inequality \eqref{GNL} under the supcritical case, which concludes Theorem \ref{the1.1}.

Furthermore, it should be stated that the Bernstein inequalities  \eqref{bern2}-\eqref{bern3} together with low-high frequency techniques as in
\cite{[Chikami]} also guarantee the following generalized Gagliardo-Nirenberg inequality in the framework of Besov-Lorentz spaces, which extends the corresponding results in \cite{[Chikami],[HMOW]}.
 \begin{theorem}\label{the1.2}
   Assume that $u \in \dot{B}_{r,\infty, \infty}^{s}\left(\mathbb{R}^{n}\right) \cap \dot{B}_{q,\infty, \infty}^{0}\left(\mathbb{R}^{n}\right) $ with $1 < q, r \leq \infty$ and $0 \leq \sigma<s<\infty.$ Then there exists a positive constant $C=C(n,q,p,r,s,\sigma)$ such that
\be\label{glgn}
\|u\|_{\dot{B}_{p,1, 1}^{\sigma}} \leq C\|u\|_{\dot{B}_{q,\infty, \infty}^{0}}^{\theta}\|u\|_{\dot{B}_{r,\infty, \infty}^{s}}^{1-\theta},
\ee
with
$$
\frac{n}{p}-\sigma=\theta \frac{n}{q}+(1-\theta)\left(\frac{n}{r}-s\right),~~0<\theta<1-\frac{\sigma}{s},~~s-\f{n}{r} \neq \sigma-\frac{n}{p}.
$$
 \end{theorem}
\begin{remark}
This theorem implies several versions of Gagliardo-Nirenberg inequalities, such as Theorem \ref{the1.1} for Lorentz spaces and \cite[Theorem 5.3]{[Chikami]} for Besov spaces. Due to \cite{[HMOW]}, it is necessary to make an assumption that $\,s-\f{n}{r} \neq \sigma-\frac{n}{p}\,$ in Theorem \ref{the1.2}.
\end{remark}
It is worth pointing out that there has existed an extensive study on Besov-Lorentz spaces (see \cite{[Wadade2],[KL],[ST],[YCP]} and references therein).
In \cite{[Wadade2]}, Wadade presented the critical Besov-Lorentz inequality \eqref{crokey1} under the case that $q>\max\{p,r\}$, and he also proved that
\be\label{Wadade}
\|u\|_{L^{q_{1}, q_{2}}(\mathbb{R}^{n} )} \leq C \|u\|_{L^{p_{1}, p_{2}}(\mathbb{R}^{n} )}^{\frac{p_{1}}{q_{1}}}\left\|\Lambda^{\frac{n}{ p_{1}}} u\right\|_{L^{p_{1}, p_{2}}(\mathbb{R}^{n} )}^{1-\frac{p_{1}}{q_{1}}},~~\text{with}~~1<p_{1}\leq q_{1}<\infty,~1\leq p_{2}\leq q_{2}\leq \infty,
\ee
which improves the classical result below due to Ozawa \cite{[Ozawa]}
\be\label{Ozawa}
\|u\|_{L^{q}(\mathbb{R}^{n} )} \leq C \|u\|_{L^{p}(\mathbb{R}^{n} )}^{\frac{p}{q}}\left\|\Lambda^{\frac{n}{ p}} u\right\|_{L^{p}(\mathbb{R}^{n} )}^{1-\frac{p}{q}},~~\text{with}~~1<p\leq q<\infty.
\ee
Combining the Besov-Lorentz inequality \eqref{crokey1} and the embedding relation in Lemma \ref{lem2.5}, we have the following corollary.
\begin{coro}
For $1<p<q<\infty,\,1<r<\infty$ and $1\leq l\leq\infty$, there exists a positive constant $C=C(n,q,p,r,l)$ such that
\begin{align}
&
\|u\|_{L^{q,l}(\mathbb{R}^{n} )}
\leq C \|u\|_{L^{p,\infty}(\mathbb{R}^{n} )}^{\f{p}{q}}  \| \Lambda^{\f{n}{r}}u\|^{1-\f{p}{q}}_{L^{r,\infty}(\mathbb{R}^{n} )},
\\
&\|u\|_{L^{q,l}(\mathbb{R}^{n} )}
 \leq C \|u\|_{L^{p,\infty}(\mathbb{R}^{n} )}^{\f{p}{q}}  \| u\|^{1-\f{p}{q}}_{\dot{F}^{\f{n}{r}}_{r,\infty,\infty}},
\\
&\|u\|_{L^{q,l}(\mathbb{R}^{n} )}
\leq C \|u\|_{L^{p,\infty}(\mathbb{R}^{n} )}^{\f{p}{q}}  \| u\|^{1-\f{p}{q}}_{\dot{B}^{\f{n}{r}}_{r,\infty,\infty}}.
\end{align}
\end{coro}
\begin{remark}
This corollary generalizes the critical interpolation inequalities \eqref{Wadade} and \eqref{Ozawa}. The results obtained here can be applied to deduce the Trudinger-Moser type inequality as in \cite{[Ozawa],[Wadade1],[Wadade2]}.
\end{remark}
Next, as an application of the generalized Gagliardo-Nirenberg inequalities \eqref{GNL} and \eqref{glgn} in Lorentz type spaces, we shall derive some new conditions for energy  equality of the Leray-Hopf weak solutions to the 3D Navier-Stokes equations. For the study of incompressible hydrodynamics equations in Lorentz spaces, we refer the reader to \cite{[Barraza],[JWW],[CF]}.

\subsection{Energy conservation in the Navier-Stokes system}
The 3D incompressible Navier-Stokes system can be written as
\be\left\{\ba\label{NS}
&v_{t} -\Delta  v+ v\cdot\ti
v  +\nabla \Pi=0, \\
&\Div v=0,\\
&v|_{t=0}=v_0,
\ea\right.\ee
 where  the unknown vector $v=v(x,t)$ describes the flow  velocity field, and the scalar function $\Pi$ represents the   pressure.
 The  initial datum $v_{0}$ is given and satisfies the divergence-free condition. It is known that regular solutions to the 3D Navier-Stokes equations \eqref{NS} satisfy the energy equality
  $$
 \|v(T)\|_{L^{2}(\mathbb{R}^{3})}^{2}+2 \int_{0}^{T}\|\nabla v\|_{L^{2}(\mathbb{R}^{3})}^{2}ds= \|v_0\|_{L^{2}(\mathbb{R}^{3})}^{2}.
 $$
However,  the   global Leray-Hopf weak solutions of the 3D Navier-Stokes equations \eqref{NS} just obey
 the energy inequality
 $$
 \|v(T)\|_{L^{2}(\mathbb{R}^{3})}^{2}+2 \int_{0}^{T}\|\nabla v\|_{L^{2}(\mathbb{R}^{3})}^{2}ds\leq \|v_0\|_{L^{2}(\mathbb{R}^{3})}^{2}.
 $$
The criteria for energy equality of the Leray-Hopf weak solutions have been established by several authors (see e.g. \cite{[Lions],[Shinbrot],[Taniuchi],[BY],[CCFS],[CL],[BC],[CCFS],[Zhang]} and references therein). We list some known results in this direction below.
A Leray-Hopf weak solution $v$ to the Navier-Stokes equations \eqref{NS} satisfies the energy equality if one of the following conditions holds
\begin{itemize}
 \item
 Lions \cite{[Lions]}:  $v\in L^{4}(0,T;L^{4}(\mathbb{R}^{3}));$
\item Shinbrot \cite{[Shinbrot]}:
\be\label{Shinb}
v\in L^{p}(0,T;L^{q}(\mathbb{R}^{3})),~\text{with}~\f{2}{p}+
 \f{2}{q}=1~\text{and}~q\geq 4;\ee
\item Taniuchi \cite{[Taniuchi]}, Beirao da Veiga-Yang \cite{[BY]}:
$v\in L^{p}(0,T;L^{q}(\mathbb{R}^{3})),$~with~$\f{2}{p}+
 \f{2}{q}=1$~and~$q\geq 4,$~or~$\f{1}{p}+
 \f{3}{q}=1$~and~$3<q< 4;$
\item  Cheskidov-Constantin-Friedlander-Shvydkoy \cite{[CCFS]}: $v\in L^{3}(0,T;B^{\f13}_{3,\infty}(\mathbb{R}^{3}));$
\item Cheskidov-Luo \cite{[CL]}: $v\in L^{\beta,\infty}(0,T;B^{\f2\beta+\f2p-1}_{p,\infty}(\mathbb{R}^{3})),$~with~$\f2p+\f1\beta<1$~and~$1\leq\beta<p\leq\infty;$
\item Berselli-Chiodaroli  \cite{[BC]}, Zhang \cite{[Zhang]}:
\be\label{bcz}
\nabla v \in L^{p}\left(0, T ; L^{q}\left(\mathbb{R}^{3}\right)\right),~\text{with}~
\frac{1}{p}+\frac{3 }{q}=2~\text{and}~\frac{3 }{2}<q<\frac{9}{5},~\text{or}~
\frac{1}{p}+\frac{6}{5 q}=1~\text{and}~q \geq \frac{9}{5}.\ee
   \end{itemize}
From the work of Cheskidov-Luo \cite{[CL]}, we see that for Shinbrot's condition \eqref{Shinb}
     the Lebesgue spaces in time direction can be replaced by Lorentz spaces. Since the Lorentz spaces $L^{r,\infty}$ are larger than the Lebesgue spaces $L^{r}$ in general,
    a natural question arises whether energy equality holds for the Leray-Hopf weak solution $v$ whose space direction belongs to Lorentz spaces. Our next main result gives a partially affirmative answer.
 \begin{theorem}\label{the1.4}
The energy equality of Leray-Hopf weak solutions $v$ to the 3D Navier-Stokes equations \eqref{NS} is valid if one of the following five conditions is satisfied
 \begin{align}
 (1)\quad& v\in L^{4}(0,T;L^{4,\infty}(\mathbb{R}^{3}));\label{EIL001} \\
 (2)\quad& v\in L^{p,\infty}(0,T;L^{q,\infty}(\mathbb{R}^{3})), ~~\text{with}~~ \f{2}{p}+
 \f{2}{q}=1, ~~q> 4;\label{EIL1}\\
 (3)\quad& v\in L^{p}(0,T;L^{q,\infty}(\mathbb{R}^{3})), ~~\text{with}~~
   \f{1}{p}+
 \f{3}{q}=1,~~ 3<q< 4;\label{EIL2}\\
 (4)\quad& \nabla v \in L^{p}\left(0, T ; L^{q,\infty}\left(\mathbb{R}^{3}\right)\right), ~~\text{with}~~
\frac{1}{p}+\frac{3 }{q}=2, ~~  \frac{3 }{2}<q<\f{9}{5};\label{EIL3}\\
 (5)\quad& \Lambda^{s} v \in L^{p}\left(0, T ; L^{q,\infty}\left(\mathbb{R}^{3}\right)\right), ~~\text{with}~~
\frac{1}{p}+\frac{6}{5 q}=\f{2s}{5}+\f{3}{5},~~s>1,~~q>1,~~\frac{1}{s}<p<3. \label{EIL4}
\end{align}
 \end{theorem}
 \begin{remark}
 Theorem
\ref{the1.4} here refines the corresponding results in \cite{[Shinbrot],[BY],[Taniuchi],[BC],[Zhang]}. It is worth pointing out that
the proof of \eqref{EIL2} and \eqref{EIL4} strongly rests on the Gagliardo-Nirenberg type
inequality \eqref{GNL} without restrictions on the second index of Lorentz norms.
 \end{remark}
\begin{remark}
Following the path in the proof of \eqref{EIL4} and using \eqref{glgn}, one can prove the following condition via Besov-Lorentz spaces for energy equality
$$
v\in L^{p}(0,T;\dot{B}^{\f{5}{2p}+\f3q-\f32}_{q,\infty,\infty}(\mathbb{R}^{3})),~~\text{with}
~~q>1,~~0<p<3~~\text{and}~~\f{5}{2p}+\f3q-\f32>\max\{1,\f{1}{p}\}.
$$
We remark that even the stronger version of this condition, with space direction in Besov spaces $\dot{B}^{\f{5}{2p}+\f3q-\f32}_{q,\infty}(\mathbb{R}^{3}), $ still improves the results involving nonhomogeneous Besov spaces in  \cite{[CL]}.
\end{remark}
\begin{remark}
According to   the boundedness of Riesz transform on Lorentz spaces,
$\nabla v $  in \eqref{EIL3} can be replaced  by its symmetrical part vorticity  $\text{curl}\,v$ or its antisymmetric part $\f{1}{2}(\nabla v-\nabla v^{^{\text{T}}}) $.
\end{remark}
\begin{remark}
To the best of authors' knowledge, it remains an open problem to show that
energy equality can be derived from the following condition
$$v\in L^{4,\infty}(0,T;L^{4,\infty}(\mathbb{R}^{3})).$$
\end{remark}
\begin{remark}
Eventually, we would like to mention that an energy conservation criterion via a combination of velocity and its gradient for the equations \eqref{NS} in $\mathbb{T}^{3}$ was recently established in \cite{[WY]}.
\end{remark}

The rest of this  paper is organized as follows. In Section 2,
we  recall  some basic  materials of various Lorentz type spaces and present embedding relation among these spaces. The generalized Young inequality and  Bernstein inequalities for Lorentz spaces are also established in this section.
Section 3 is devoted to the proof of Theorem \ref{the1.1} and Theorem \ref{the1.2}.
Finally, as an application of the above two theorems, we prove Theorem \ref{the1.4} in Section 4, which gives several new criteria for
energy conservation of 3D Navier-Stokes equations in Lorentz spaces.
\section{Notations and  key auxiliary lemmas} \label{section2}
\subsection{Lorentz spaces and generalized Bernstein inequality}
Throughout this paper, we will use the summation convention on repeated indices. $C$ will denote positive absolute constants which may be different from line to line unless otherwise stated in this paper. $a\approx b$ means that $C^{-1}b\leq a\leq Cb$ for some constant $C>1$.
 $\chi_{\Omega}$ stands for the characteristic function of a set $\Omega\subset \mathbb{R}^{n}$. $|E|$ represents the $n$-dimensional Lebesgue measure of a set $E\subset \mathbb{R}^{n}$. Let $\mathcal{M}$ be the Hardy-Littlewood maximal operator and its definition is given by
  $$\mathcal{M}f(x)=\sup_{r>0}\f{1}{|B(r)|}\int_{B(r)}|f(x-y)|dy,
  $$
where $f$ is any locally integrable function on $\mathbb{R}^{n}$, and $B(r)$ is the open ball centered at the origin with radius $r>0$.

  Next, we present some basic facts on Lorentz spaces. Recall that the distribution function of a   measurable function  $f$ on $\Omega$ is the function $f_{\ast}$ defined on $[0,\infty)$ by
$$
f_{\ast}(\alpha)=|\{x\in \Omega:|f(x)|>\alpha\}|.
$$
The decreasing rearrangement of $f$ is the function $f^{\ast}$ defined on $[0,\infty)$ by
$$
f^{\ast}(t)=\inf\{\alpha>0: f_{\ast}(\alpha)\leq t\}.
$$
For $p,q\in (0,\infty]$, we define
$$
\|f\|_{L^{p,q}(\Omega)}=\left\{\ba
&\B(\int_{0}^{\infty}\left(t^{\f{1}{p}}f^{\ast}(t)\right)^{q}\f{dt}{t}\B)^{\f1q}, ~~\text { if } q<\infty, \\
 &\sup_{t>0}t^{\f{1}{p}}f^{\ast}(t), ~~\text { if } q=\infty.
\ea\right.
$$
Furthermore,
$$
L^{p,q}(\Omega)=\big\{f: f~ \text{is a measurable function on}~ \Omega ~\text{and} ~\|f\|_{L^{p,q}(\Omega)}<\infty\big\},
$$
which implies that $L^{\infty,\infty}=L^{\infty}$, $L^{q,q}=L^{q}$ and $L^{\infty,q}=\{0\}$	for $0<q<\infty$.		

Notice that identity definition of  Lorentz norm  can be found in \cite{[Grafakos],[Maly]}. Indeed, for $0<p\leq\infty$ and $0<q\leq\infty$, there holds
$$
\|f\|_{L^{p,q}(\Omega)}= p^{\f{1}{q}}\|\| \alpha \chi_{(\alpha, \infty)}(|f(\cdot)|)\left\|_{L^{p}(\Omega)}\right\|_{L^{q}\left(\mathbb{R}^{+}, \frac{d \alpha}{\alpha}\right)}\\
= \left\{\ba
&\B(p\int_{0}^{\infty}\alpha^{q}f_{*}(\alpha)^{\f{q}{p}}\f{d\alpha}{\alpha}\B)^{\f{1}{q}} , ~~\text { if } q<\infty, \\
 &\sup_{\alpha>0}\alpha f_{*}(\alpha)^{\f{1}{p}} ,~~\text { if } q=\infty.
\ea\right.
$$
Similarly, one can define
Lorentz spaces $L^{p,q}(0,T;X)$ in time for $0<p, q\leq\infty$. $f\in L^{p,   q}(0,T;X)$ means that $\|f\|_{L^{p,q}(0,T;X)}<\infty$, where
$$\|f\|_{L^{p,q}(0,T;X)}=\left\{\ba
&\B(p\int_{0}^{\infty}\alpha^q|\{t\in[0,T)
:\|f(t)\|_{X}>\alpha\}|^{\f{q}{p}}\f{d\alpha}{\alpha}\B)^{\f{1}{q}} , ~~\text { if } q<\infty, \\
 &\sup_{\alpha>0}\alpha|\{t\in[0,T)
:\|f(t)\|_{X}>\alpha\}|^{\f{1}{p}} ,~~\text { if } q=\infty.\ea\right.
$$

Note that the triangle inequality is not valid for $\|\cdot\|_{L^{p, q}(\mathbb{R}^{n})}  .$ Another equivalent norm in Lorentz spaces is defined as\be\label{thirddefin}
\|f\|^{\ast}_{L^{p,q }(\mathbb{R}^{n} )}=\left\{\ba
&\left( \int_{0}^{\infty}\left(t^{\frac{1}{p}} f^{* *}(t)\right)^{q} \f{dt}{t}\right)^{\frac{1}{q}}, ~~\text { if } 1<p<\infty, 1 \leq q<\infty, \\
&\sup _{t>0} t^{\frac{1}{p}}f^{* *}(t), ~~\text { if } 1<p \leq \infty, q=\infty, \ea
 \right.
\ee
where
$$
f^{* *}(t)=\frac{1}{t} \int_{0}^{t} f^{*}(s) d s =\sup _{|E| \geq t}\left(\frac{1}{|E|} \int_{E}|f(x)| d x\right),~ t>0.
$$
 In addition, Lorentz spaces  endowed with the norm $\|\cdot\|^{\ast}_{L^{p, q} }$ are Banach spaces, and there holds
\be\label{thirddefincoro}
\|f\|_{L^{p, q} (\mathbb{R}^{n} )} \leq\|f\|_{L^{p, q} (\mathbb{R}^{n} )}^{*} \leq \frac{p}{p-1}\|f\|_{L^{p, q}(\mathbb{R}^{n} ) }.
\ee
Most of the above statement is borrowed from  \cite{[Grafakos],[BL],[CF]}.

Subsequently, we present norm-equivalence concerning Lorentz spaces.
\begin{lemma}
Let $f$ be in $L^{p,q}(\mathbb{R}^{n})$ with $1<p\leq\infty$ and $1\leq q\leq\infty$. Then there holds
\be\label{inorm}
\| f\|_{L^{p,q}(\mathbb{R}^{n} )}\leq C_{1}\| f\|^{\ast}_{L^{p,q}(\mathbb{R}^{n} )} \leq C_{2} \|\mathcal{M}f\|_{L^{p,q}(\mathbb{R}^{n} )} \leq C_{3}\| f\|^{\ast}_{L^{p,q}(\mathbb{R}^{n} )}\leq C_{4}
\| f\|_{L^{p,q}(\mathbb{R}^{n} )},\ee
where $C_{1},C_{2},C_{3}$ and $C_{4}$ are positive constants depending only on $p,q$ and $n$.
\end{lemma}
\begin{proof}
Since $f\in L^{p,q}(\mathbb{R}^{n})$ with $1<p\leq\infty$ and $1\leq q\leq\infty$, it follows from \eqref{Inclusion} that $f$ is a locally integrable function on $\mathbb{R}^{n}$. Recall the pointwise inequality involving Hardy-littlewood maximal operator (see \cite[p.41]{[Duoandikoetxea]} and \cite[Chapter 3]{[BS]}) below,
$$
c_{1}(\mathcal{M}f)^{\ast}(t)\leq f^{\ast\ast}(t)\leq c_{2} (\mathcal{M}f)^{\ast}(t),~t>0,
$$
where $c_{1}$ and $c_{2}$ are positive constants depending only on $n$. This together with \eqref{thirddefin} means
$$
  \| f\|^{\ast}_{L^{p,q}(\mathbb{R}^{n} )} \leq C  \|\mathcal{M}f\|_{L^{p,q}(\mathbb{R}^{n} )} \leq C \| f\|^{\ast}_{L^{p,q}(\mathbb{R}^{n} )}.
$$
The conclusion is a straightforward consequence of the latter and \eqref{thirddefincoro}.
\end{proof}
\begin{remark} Even if $0<q<1$, the equivalent relation $\|\mathcal{M}f\|_{L^{p,q}(\mathbb{R}^{n})}\thickapprox \|f\|_{L^{p,q}(\mathbb{R}^{n})}$ still holds for all functions $f\in L^{p,q}(\mathbb{R}^{n})$ with $1<p\leq\infty$. Indeed, thanks to Marcinkiewicz's interpolation theorem for Lorentz spaces \cite[Theorem 1.4.19]{[Grafakos]}, it follows from the fact that Hardy-littlewood maximal operator $\mathcal{M}$ is a sublinear operator of both weak type $(1,1)$ and strong type $(\infty,\infty)$ that $\mathcal{M}$ is also bounded on $L^{p,q}(\mathbb{R}^{n})$ for any $p\in(1,\infty]$ and $q\in(0,\infty]$, which yields that $\|\mathcal{M}f\|_{L^{p,q}(\mathbb{R}^{n})}\leq C(n,p,q)\|f\|_{L^{p,q}(\mathbb{R}^{n})}$ for all functions $f\in L^{p,q}(\mathbb{R}^{n})$. On the other hand, Lebesgue's differentiation theorem implies that $|f(x)|\leq \mathcal{M}f(x)$ for almost all $x\in \mathbb{R}^{n}$, which yields that $f_{\ast}\leq (\mathcal{M}f)_{\ast}$ and $\|f\|_{L^{p,q}(\mathbb{R}^{n})}\leq \|\mathcal{M}f\|_{L^{p,q}(\mathbb{R}^{n})}$ for $1<p\leq\infty$ and $0<q\leq\infty$.
\end{remark}

We list the properties of Lorentz spaces as follows.
\begin{itemize}

\item Interpolation characteristic of Lorentz spaces \cite{[BL],[DDN]}
\be\label{Interpolation characteristic}
\ba &\|f \|_{L^{p,p_{1}}(\mathbb{R}^{n})} \leq  \left[\frac{(r-q)p^{2}}{(r-p)(p-q)p_{1}}\right]^{\frac{1}{p_{1}}}\|f \|_{L^{q,\infty}(\mathbb{R}^{n})}^{\alpha} \|f \|_{L^{r ,\infty}(\mathbb{R}^{n})}^{1-\alpha},\\
&\text{with}~~ \f{1}{p}=\f{\alpha}{q}
+\f{1-\alpha}{r},~0<\alpha<1,~0<q<p<r\leq \infty ~~ \text{and}~~ 0<p_{1}\leq \infty.\ea\ee

\item

$$\||f|^{\lambda} \|_{L^{p,q}(\mathbb{R}^{n})}= \|f \|^{\lambda}_{L^{\lambda p,\lambda q}(\mathbb{R}^{n})},~~\text{with}~~
0< \lambda<\infty~~\text{and}~~0<p,q\leq\infty.$$

\item
H\"older's inequality in Lorentz spaces  \cite{[Neil]}
\be\label{HolderIQ}
\ba &\|fg\|_{L^{r,s}(\Omega)}\leq C(r_{1},r_{2},s_{1},s_{2})\,\|f\|_{L^{r_{1},s_{1}}(\Omega)}\|g\|_{L^{r_{2},s_{2}}(\Omega)},
\\
&\text{with}~~\f{1}{r}=\f{1}{r_{1}}+\f{1}{r_{2}},~~\f{1}{s}=\f{1}{s_{1}}+\f{1}{s_{2}},~~0<r_{1},r_{2},s_{1},s_{2}\leq \infty.\ea\ee

\item

The Lorentz spaces increase as the exponent $q$ increases \cite{[Grafakos],[Maly]}

For $0< p\leq\infty$ and $0< q_{1}<q_{2}\leq\infty,$
\be\label{lincreases}
\|f\|_{L^{p,q_{2}}(\Omega)}\leq \B(\f{q_{1}}{p}\B)^{\f{1}{q_{1}}-\f{1}{q_{2}}}\|f\|_{L^{p,q_{1}}(\Omega)}.
\ee

\item
Inclusion in Lorentz spaces on bounded domains  \cite{[Grafakos],[Maly]}

For any $1\leq m<M\leq\infty$, $~~1\leq r,q\leq\infty$,
\be\label{Inclusion}
\|f\|_{L^{m,r}(\Omega)}\leq \B(\f{1}{m}\B)^{\f{r-1}{r}}
\B(\f{q}{M}\B)^{\f{1}{q}}\f{|\Omega|
^{\f{1}{m}-\f{1}{M}}}{\f{1}{m}-\f{1}{M}}
\|f\|_{L^{M,q}(\Omega)}.
\ee

\item Sobolev inequality in Lorentz spaces \cite{[Neil],[Tartar]}
\be\label{sl}
\|f\|_{L^{\f{np}{n-p},p}(\mathbb{R}^{n})}\leq C(n,p)\,\|\nabla f\|_{L^{p}(\mathbb{R}^{ n })}~~\text{with}~~1\leq  p <n.\ee
\item Young inequality in Lorentz spaces \cite{[Neil]}

Let $1<p,q,r<\infty$, $0<s_{1},s_{2}\leq\infty$
,$\f{1}{p}+\f{1}{q}=\f{1}{r}+1$, and $ \f{1}{s}=\f{1}{s_{1}}+\f{1}{s_{2}}$. Then there holds
\be\label{young}
\|f\ast g\|_{L^{r,s}(\mathbb{R}^{n})}\leq C(p,q,s_{1},s_{2})\,\|f \|_{L^{p,s_{1}}(\mathbb{R}^{n})}\|g \|_{L^{q,s_{2}}(\mathbb{R}^{n})}.
\ee
\end{itemize}

It should be mentioned that the classical Young inequality \eqref{young} due to O'Neil requires the first index of every Lorentz norm is larger than 1. Grafakos \cite{[Grafakos]} improved O'Neil's result and showed that
$$
\|f\ast g\|_{L^{q,\infty}(\mathbb{R}^{n})}\leq C \|f\|_{L^{r,\infty}(\mathbb{R}^{n})}\|g\|_{L^{p}(\mathbb{R}^{n})},~~\text{with}~~
\f1q+1=\f1p+\f1r,\, 1\leq p<\infty~~\text{and}~~1<q,r<\infty.
$$
The following lemma extends it to a more general version, which we shall give a different proof from that of \cite[Theorem 1.2.13]{[Grafakos]}.
\begin{lemma}\label{YoungGen}
Suppose that $0<l\leq s \leq\infty$, $1 \leq r<\infty$ and $1<p, q<\infty$. If $f \in L^{q, l}(\mathbb{R}^{n})$ and $g \in L^{r}(\mathbb{R}^{n})$ with
\be\label{YoungGen 1}
\frac{1}{p}+1=\frac{1}{q}+\frac{1}{r},
\ee
then $f \ast g \in L^{p, s}(\mathbb{R}^{n})$, and there exists a positive constant $C$ depending only on $r,q,s$ and $l$ such that
\be\label{YoungGen 2}
\|f \ast g\|_{L^{p, s}(\mathbb{R}^{n} )} \leq C\|f\|_{L^{q, l}(\mathbb{R}^{n} )}\|g\|_{L^{r}(\mathbb{R}^{n} )}.
\ee
\end{lemma}
\begin{remark}
In general, \eqref{YoungGen 2} fails when $l>s$. Indeed, here is a counterexample as follows: for any $f \in L^{q, l}(\mathbb{R}^{n})$ with $0<s<l\leq\infty$ and $1<q<\infty$, it follows from \eqref{Inclusion} and \eqref{lincreases} that $f$ is a locally integrable function on $\mathbb{R}^{n}$. Take $g=\chi_{Q}$, where $Q=\left\{(y_{1},y_{2},\ldots,y_{n})\in \mathbb{R}^{n}: \max_{i}|y_{i}|\leq 1/2 \right\}$ is a cube in $\mathbb{R}^{n}$. Let $(\chi_{Q})_{r}(x)=r^{-n}\chi_{Q}(x/r)$ for all $x\in \mathbb{R}^{n}$ and $r>0$. Then Lebesgue's differentiation theorem guarantees that $\lim\limits_{m\to\infty} f \ast (\chi_{Q})_{\frac{1}{m}}(x)=f(x)$ for almost all $x\in \mathbb{R}^{n}$, which together with \cite[Proposition 1.4.5]{[Grafakos]} implies that $f^{\ast}\leq \liminf\limits_{m\to\infty} \left(f \ast (\chi_{Q})_{\frac{1}{m}}\right)^{\ast}$. Hence we may apply Fatou's lemma and \eqref{YoungGen 2} with $r=1$ to derive that for $0<s<l\leq\infty$ and $1<q<\infty$,
$$ \|f \|_{L^{q, s}(\mathbb{R}^{n} )} \leq \liminf\limits_{m\to\infty}\|f \ast (\chi_{Q})_{\frac{1}{m}}\|_{L^{q, s}(\mathbb{R}^{n} )} \leq C\|f\|_{L^{q, l}(\mathbb{R}^{n} )}\|\chi_{Q}\|_{L^{1}(\mathbb{R}^{n} )}=C\|f\|_{L^{q, l}(\mathbb{R}^{n} )}.$$
This contradicts the fact that $L^{q, s}(\mathbb{R}^{n})\subsetneqq L^{q, l}(\mathbb{R}^{n})$. Additionally, we remark that necessity of the condition \eqref{YoungGen 1} results from dilation structure of the convolution $f \ast g$ in \eqref{YoungGen 2}, and the exclusion of endpoint cases for the three indices $q,\,r$ and $p$ is due to \cite[Example 1.2.14]{[Grafakos]}.
\end{remark}

The proof of this lemma relies on Marcinkiewicz's interpolation theorem for Lorentz spaces as follows, which is given in \cite[Theorem 1.4.19]{[Grafakos]}.
\begin{lemma}\label{gw}
Let $0<r \leq \infty, 0<p_{0} \neq p_{1} \leq \infty,$ and $0<q_{0} \neq q_{1} \leq \infty$ and let $(X, \mu)$ and $(Y, v)$ be two measure spaces. Let $T$ be either a quasilinear operator with some constant $K>0$ defined on $L^{p_{0}}(X)+L^{p_{1}}(X)$ and taking values in the set of measurable functions on $Y$ or a linear operator defined on the set of simple functions on $X$ and taking values as before. Assume that for some $M_{0}, M_{1}<\infty$ the following (restricted) weak type estimates hold:
$$
\begin{array}{l}
\left\|T\left(\chi_{A}\right)\right\|_{L^{q_{0}, \infty}} \leq M_{0} \mu(A)^{1 / p_{0}}, \\
\left\|T\left(\chi_{A}\right)\right\|_{L^{q_{1}, \infty}} \leq M_{1} \mu(A)^{1 / p_{1}},
\end{array}
$$
for all measurable subsets $A$ of $X$ with $\mu(A)<\infty .$ Fix $0<\theta<1$ and let
$$
\frac{1}{p}=\frac{1-\theta}{p_{0}}+\frac{\theta}{p_{1}} \quad \text { and } \quad \frac{1}{q}=\frac{1-\theta}{q_{0}}+\frac{\theta}{q_{1}}.
$$
Then there exists a positive constant $C,$ which depends only on $K, p_{0}, p_{1}, q_{0}, q_{1}, r$ and $\theta,$ such that for all functions $f$ in the domain of $T$ and in $L^{p, r}(X)$ we have
$$
\|T(f)\|_{L^{q, r}} \leq C(M_{0}+M_{1})\|f\|_{L^{p, r}}.
$$
\end{lemma}

Now we continue with the proof of Lemma \ref{YoungGen}.
\begin{proof}[Proof of Lemma \ref{YoungGen}]
Since \eqref{lincreases} implies $L^{p, l}(\mathbb{R}^{n}) \hookrightarrow L^{p, s}(\mathbb{R}^{n})$, it suffices to prove \eqref{YoungGen 2} for the case when $s=l$.

To this end, fix $g\in L^{r}(\mathbb{R}^{n})$. Let $T(f)=f\ast g$, then $T$ is a linear operator defined on the set of simple functions on $\mathbb{R}^{n}$. For all measurable subsets $A$ of $\mathbb{R}^{n}$ with $|A|<\infty$, it follows from \eqref{lincreases} and Young's inequality for Lebesgue spaces that
$$
\begin{array}{l}
\left\|T\left(\chi_{A}\right)\right\|_{L^{r, \infty}(\mathbb{R}^{n} )}\leq\left\|T\left(\chi_{A}\right)\right\|_{L^{r}(\mathbb{R}^{n} )} \leq  \|g\|_{L^{r}(\mathbb{R}^{n})}\|\chi_{A}\|_{L^{1}(\mathbb{R}^{n} )}= \|g\|_{L^{r}(\mathbb{R}^{n})}|A|, \\
\left\|T\left(\chi_{A}\right)\right\|_{L^{\infty, \infty}(\mathbb{R}^{n} )}=\left\|T\left(\chi_{A}\right)\right\|_{L^{\infty}(\mathbb{R}^{n} )} \leq  \|g\|_{L^{r}(\mathbb{R}^{n})}\|\chi_{A}\|_{L^{r'}(\mathbb{R}^{n} )} = \|g\|_{L^{r}(\mathbb{R}^{n})} |A|^{1 /r'},
\end{array}
$$
where $1<r'\leq\infty$ and $1 /r + 1 /r'=1.$

Take $\theta=1-r/p=r(1-1/q)\in (0,1)$. Then the hypotheses on the indices imply that
$$
\frac{1}{q}=\frac{1-\theta}{1}+\frac{\theta}{r'} \quad \text { and } \quad \frac{1}{p}=\frac{1-\theta}{r}+\frac{\theta}{\infty}.
$$
With the help of Lemma \ref{gw}, we obtain that for all functions $f\in L^{q,l}(\mathbb{R}^{n})$ with $l\in (0,\infty]$ there holds
$$\|f\ast g\|_{L^{p,l}(\mathbb{R}^{n} )}=\|T(f)\|_{L^{p,l}(\mathbb{R}^{n} )}\leq C\|g\|_{L^{r}(\mathbb{R}^{n})}\|f\|_{L^{q,l}(\mathbb{R}^{n})}.
$$
Here $C>0$ depends only on $r,q$ and $l$. This completes the proof.
\end{proof}

As an application of this lemma, we may derive the generalized Bernstein inequality for Lorentz spaces as follows.
\begin{lemma}\label{lem2.4}
Let a ball $B=\left\{\xi \in \mathbb{R}^{n}: |\xi| \leq R\right\}$ with $0<R<\infty$ and an annulus $\mathcal{C}=\left\{\xi \in \mathbb{R}^{n}: r_{1} \leq|\xi| \leq r_{2}\right\}$ with $0<r_{1}<r_{2}<\infty$. Then a positive constant $C$ exists such that for any nonnegative integer $k,$ any couple $(p, q)$ with $1< p <q<\infty,$ any $\lambda\in (0, \infty)$, and any function $u$ in $L^{p,\infty}(\mathbb{R}^{n})$ or in $L^{q,l}(\mathbb{R}^{n})$ with $0< l\leq\infty$, there hold
\begin{align}
&  \sup _{|\alpha|=k}\left\|\partial^{\alpha} u\right\|_{L^{\infty}(\mathbb{R}^{n})} \leq C \lambda^{k+\frac{n}{p}}\|u\|_{L^{p,\infty}(\mathbb{R}^{n})} ~~ \text{with}~~ \operatorname{supp} \widehat{u} \subset \lambda B ~~ \text{and}~~1< p \leq \infty\,;\label{bern1}\\
&  \sup _{|\alpha|=k}\left\|\partial^{\alpha} u\right\|_{L^{q,1}(\mathbb{R}^{n})} \leq C \lambda^{k+n\left(\frac{1}{p}-\frac{1}{q}\right)}\|u\|_{L^{p,\infty}(\mathbb{R}^{n})} ~~ \text{with}~~ \operatorname{supp} \widehat{u} \subset \lambda B\,;\label{bern2}\\
&  \sup _{|\alpha|=k}\left\|\partial^{\alpha} u\right\|_{L^{q,l}(\mathbb{R}^{n})} \leq C \lambda^{k}\|u\|_{L^{q,l}(\mathbb{R}^{n})} ~~ \text{with}~~ \operatorname{supp} \widehat{u} \subset \lambda B\,;\label{bern3}\\
&  C^{-1} \lambda^{k}\|u\|_{L^{q,l}(\mathbb{R}^{n})} \leq \sup _{|\alpha|=k}\left\|\partial^{\alpha} u\right\|_{L^{q,l}(\mathbb{R}^{n})} \leq C \lambda^{k}\|u\|_{L^{q,l}(\mathbb{R}^{n})} ~~ \text{with}~~ \operatorname{supp} \widehat{u} \subset \lambda \mathcal{C}.\label{bern4}
\end{align}
Here $\lambda B=\left\{\xi \in \mathbb{R}^{n}: |\xi| \leq \lambda R\right\}$ and $\lambda \mathcal{C}=\left\{\xi \in \mathbb{R}^{n}: \lambda r_{1} \leq|\xi| \leq \lambda r_{2}\right\}$.
\end{lemma}
\begin{remark}
 This lemma  extends the following result due to McCormick-Robinson-Rodrigo in \cite{[MRR]}, for $\operatorname{supp} \widehat{f} \subset \lambda B$,
$$
\|f\|_{L^{p}(\mathbb{R}^{n})} \leq c \lambda^{n(1 / q-1 / p)}\|f\|_{L^{q, \infty}(\mathbb{R}^{n})},~~ 1<q<p\leq\infty.
$$
\end{remark}
\begin{remark}
In this lemma, the necessity of the assumption that $p,q>1$ results from Young inequality \eqref{YoungGen 2} and \eqref{young} for Lorentz spaces.
\end{remark}
\begin{proof} (1)\,\,Let $\psi$ be a Schwartz function on $\mathbb{R}^{n}$ such that $\chi_{B}\leq \hat{\psi}\leq \chi_{2B}$. Since $\widehat{u}(\xi)=\hat{\psi}(\xi/\lambda) \widehat{u}(\xi)$ when $\operatorname{supp} \widehat{u} \subset \lambda B$, we have
$$
\partial^{\alpha} u=i^{|\alpha|}\mathcal{F}^{-1}(\xi^{\alpha}\hat{u})=
i^{|\alpha|}\mathcal{F}^{-1}(\xi^{\alpha}\hat{\psi}(\xi/\lambda) \widehat{u}(\xi))=
\lambda^{|\alpha|}(\partial^{\alpha} \psi)_{\lambda}\ast u.
$$
Here  $(\partial^{\alpha} \psi)_{\lambda}(x)=\lambda^{n}\,\partial^{\alpha} \psi(\lambda x)$ for all $x\in\mathbb{R}^{n}$.

Fix $|\alpha|=k$. From the H\"older inequality \eqref{HolderIQ} for Lorentz spaces, we infer that
$$\ba
\|\partial^{\alpha}u \|_{L^{\infty}(\mathbb{R}^{n})}\leq& \lambda^{k}\sup _{x\in\mathbb{R}^{n}}\int_{\mathbb{R}^{n}} |(\partial^{\alpha} \psi)_{\lambda}(x-y)||u(y)|dy\\
\leq& C\lambda^{k} \sup _{x\in\mathbb{R}^{n}}\|(\partial^{\alpha} \psi)_{\lambda}(x-\cdot)\|_{L^{\f{p}{p-1},1}(\mathbb{R}^{n}) }\|u\|_{L^{p,\infty}(\mathbb{R}^{n})}\\=&C\lambda^{k+\f{n}{p}} \|\partial^{\alpha} \psi  \|_{L^{\f{p}{p-1},1}(\mathbb{R}^{n}) }\|u\|_{L^{p,\infty}(\mathbb{R}^{n})},
\ea$$
where $C=C(p)>0$.
Note that  $\partial^{\alpha} \psi\in \mathcal{S}(\mathbb{R}^{n})$, then for any $\varepsilon>0$, there exists a constant $C=C(\varepsilon,n,\partial^{\alpha} \psi)>0$ such that
$0\leq (\partial^{\alpha} \psi)_{\ast}(s)\leq Cs^{-\varepsilon}$ for all $s>0$. This yields that
$$ \ba
\|\partial^{\alpha} \psi \|_{L^{\f{p}{p-1},1}(\mathbb{R}^{n})} =\f{p}{p-1}\int_{0}^{\infty}(\partial^{\alpha} \psi)_{\ast}^{\f{p-1}{p}}(s)ds \leq C(\int_{0}^{1}s^{-\f12}ds+\int_{1}^{\infty}s^{-2}ds)<\infty,
\ea$$
which implies \eqref{bern1}.

(2)\,\,Take $1/r=1+1/q-1/p$, then the hypotheses on the indices imply that $1<r<q<\infty$. In light of \eqref{young}, we see that there exists a positive constant $C=C(p,q)$ such that
$$\ba
\|\partial^{\alpha}u \|_{L^{q,1}(\mathbb{R}^{n})}=& \lambda^{k}\left\|(\partial^{\alpha} \psi)_{\lambda}\ast u\right\|_{L^{q,1}(\mathbb{R}^{n})}\\
\leq & C\lambda^{k}\left\|(\partial^{\alpha} \psi)_{\lambda}\right\|_{L^{r,1}(\mathbb{R}^{n})}\|u\|_{L^{p,\infty}(\mathbb{R}^{n})}\\
=& C\lambda^{k+n\left(1-\frac{1}{r}\right)}\left\|\partial^{\alpha} \psi\right\|_{L^{r,1}(\mathbb{R}^{n})}\|u\|_{L^{p,\infty}(\mathbb{R}^{n})}\\
=& C\lambda^{k+n\left(\frac{1}{p}-\frac{1}{q}\right)}\left\|\partial^{\alpha} \psi\right\|_{L^{r,1}(\mathbb{R}^{n})}\|u\|_{L^{p,\infty}(\mathbb{R}^{n})}.
\ea$$
By a similar argument used in the proof of \eqref{bern1}, we may derive that $\left\|\partial^{\alpha} \psi\right\|_{L^{r,1}(\mathbb{R}^{n})}<\infty$. This implies \eqref{bern2}.

(3)\,\,As $\partial^{\alpha} \psi\in \mathcal{S}(\mathbb{R}^{n})\subset L^{1}(\mathbb{R}^{n})$, Lemma \ref{YoungGen} enable us to deduce that
$$\ba
\|\partial^{\alpha}u \|_{L^{q,l}(\mathbb{R}^{n})}=& \lambda^{k}\left\|(\partial^{\alpha} \psi)_{\lambda}\ast u\right\|_{L^{q,l}(\mathbb{R}^{n})}\\
\leq & C\lambda^{k}\left\|(\partial^{\alpha} \psi)_{\lambda}\right\|_{L^{1}(\mathbb{R}^{n})}\|u\|_{L^{q,l}(\mathbb{R}^{n})}\\
=& C\lambda^{k}\left\|\partial^{\alpha} \psi\right\|_{L^{1}(\mathbb{R}^{n})}\|u\|_{L^{q,l}(\mathbb{R}^{n})},
\ea$$
where $C=C(q,l)>0$. This implies \eqref{bern3}.

(4)\,\,Observe that \eqref{bern3} implies the second inequality of \eqref{bern4}, it suffices to show the first inequality in \eqref{bern4}. Let $\eta$ be a Schwartz function on $\mathbb{R}^{n}$ such that $\chi_{\mathcal{C}}\leq \hat{\eta}\leq \chi_{\tilde{\mathcal{C}}}$, where $\tilde{\mathcal{C}}=\left\{\xi \in \mathbb{R}^{n}:  r_{1}/2 \leq|\xi| \leq 2 r_{2}\right\}$. It follows from $ \operatorname{supp} \hat{u} \subset\lambda \mathcal{C}$ that for all $\xi\in\mathbb{R}^{n}$,
$$
\hat{u}(\xi)=\sum_{|\alpha|=k} \frac{(-i \xi)^{\alpha}}{|\xi|^{2 k}} \hat{\eta}\left(\xi/\lambda\right)(i \xi)^{\alpha} \hat{u}(\xi)=\lambda^{-k} \sum_{|\alpha|=k} \frac{(-i \xi / \lambda)^{\alpha}}{|\xi/ \lambda|^{2 k}} \hat{\eta}\left(\xi/\lambda\right)\mathcal{F}(\partial^{\alpha}u)(\xi).  $$
Therefore, we may write
$$
u=\lambda^{-k} \sum_{|\alpha|=k} (g_{\alpha})_{\lambda} \ast \partial^{\alpha} u,
$$
where $(g_{\alpha})_{\lambda}(x)=\lambda^{n} g_{\alpha}(\lambda x)$ for all $x\in\mathbb{R}^{n}$, and
$$
 g_{\alpha}=\mathcal{F}^{-1}\left(\frac{(-i \xi)^{\alpha}}{|\xi|^{2 k}} \hat{\eta}(\xi)\right) \in \mathcal{S}(\mathbb{R}^{n})\subset L^{1}(\mathbb{R}^{n}).
$$
This together with Lemma \ref{YoungGen} yields that a constant $C=C(n,q,l,k)>0$ exists such that
\begin{align*}
\|u\|_{L^{q,l}(\mathbb{R}^{n})} & \leq C\lambda^{-k}\sum_{|\alpha|=k}\left\|(g_{\alpha})_{\lambda}\right\|_{L^{1}(\mathbb{R}^{n})}\left\|\partial^{\alpha} u\right\|_{L^{q,l}(\mathbb{R}^{n})} \\
 &\leq C\lambda^{-k}\left(\sum_{|\alpha|=k}\left\|g_{\alpha}\right\|_{L^{1}(\mathbb{R}^{n})}\right) \sup _{|\alpha|=k}\left\|\partial^{\alpha} u\right\|_{L^{q,l}(\mathbb{R}^{n})}.
\end{align*}
This concludes \eqref{bern4} and the proof is complete.
\end{proof}
\subsection{Besov-Lorentz spaces, Sobolev-Lorentz spaces and Triebel-Lizorkin-Lorentz spaces}
  $\mathcal{S}$ denotes the Schwartz class of rapidly decreasing functions, $\mathcal{S}'$ the
space of tempered distributions, $\mathcal{S}'/\mathcal{P}$ the quotient space of tempered distributions which modulo polynomials.
  We use $\mathcal{F}f$ or $\widehat{f}$ to denote the Fourier transform of a tempered distribution $f$.
To define Besov-Lorentz spaces, we need the following dyadic unity partition
(see e.g. \cite{[BCD]}). Choose two nonnegative radial
functions $\varrho$, $\varphi\in C^{\infty}(\mathbb{R}^{n})$ be
supported respectively in the ball $\{\xi\in
\mathbb{R}^{n}:|\xi|\leq \frac{4}{3} \}$ and the shell $\{\xi\in
\mathbb{R}^{n}: \frac{3}{4}\leq |\xi|\leq
  \frac{8}{3} \}$ such that
\begin{equation*}
 \varrho(\xi)+\sum_{j\geq 0}\varphi(2^{-j}\xi)=1, \quad
 \forall\xi\in\mathbb{R}^{n}; \qquad
 \sum_{j\in \mathbb{Z}}\varphi(2^{-j}\xi)=1, \quad \forall\xi\neq 0.
\end{equation*}
The nonhomogeneous dyadic
blocks $\Delta_{j}$ are defined by
$$
\ba
\Delta_{j} u:=0 ~~ \text{if} ~~ j \leq-2, ~~ \Delta_{-1} u:=\varrho(D) u ,~~\Delta_{j} u:=\varphi\left(2^{-j} D\right) u  ~~\text{if}~~ j \geq 0,~~
S_{j}u:= \sum_{k\leq j-1}\Delta_{k}u.
\ea
$$
The homogeneous Littlewood-Paley operators are defined as follows
\begin{equation*}
  \forall j\in\mathbb{Z},\quad \dot{\Delta}_{j}f:= \varphi(2^{-j}D)f ~~\text { and }~~ \dot{S}_{j}f:= \sum_{k\leq j-1}\dot{\Delta}_{k}f.
\end{equation*}
The homogeneous Besov-Lorentz   space $ \dot{B}^{s}_{p,q,r}(\mathbb{R}^{n})$ is the set of $f\in\mathcal{S}'(\mathbb{R}^{n})/\mathcal{P}(\mathbb{R}^{n})$ such that
\begin{equation*}
  \norm{f}_{\dot{B}^{s}_{p,q,r}}:=\norm{\left\{2^{js}\norm{\dot{\Delta}
  _{j}f}_{L^{p,q}(\mathbb{R}^{n})}\right\}}_{\ell^{r}(\mathbb{Z})}<\infty.
\end{equation*}
Here $\ell^{r}(\mathbb{Z})$ represents the set of sequences with summable $r$-th powers.
The homogeneous Sobolev-Lorentz norm $\|\cdot\|_{\dot{H}^{s}_{p,p_{1}}(\mathbb{R}^{n})}$  is defined as $\|f\|_{\dot{H}^{s}_{p,p_{1}}(\mathbb{R}^{n})}=\|\Lambda^{s}f\|_{L^{p,p_{1}}(\mathbb{R}^{n})}.$
When $p_{1}=p$, the Sobolev-Lorentz spaces reduce to the classical Sobolev spaces $\dot{H}^{s}_{p}(\mathbb{R}^{n})$.

For $p,q,r\in (0,\infty]$ and $s \in \mathbb{R},$ the homogeneous Triebel-Lizorkin-Lorentz space $\dot{F}_{p, q,r}^{s}(\mathbb{R}^{n})$ is defined by
$$
\dot{F}_{p, q,r}^{s}(\mathbb{R}^{n})=\left\{f \in \mathcal{S}'(\mathbb{R}^{n})/\mathcal{P}(\mathbb{R}^{n}) : \|f\|_{\dot{F}_{p, q,r}^{s}}<\infty\right\}.
$$
Here
$$
\|f\|_{\dot{F}_{p, q,r}^{s}}:= \left\{  \ba
 &\| \{\sum_{j=-\infty}^{\infty} (2^{j s } |\dot{\Delta}_{j} f |)^{r}\}^{\frac{1}{r}} \|_{L^{p,q}(\mathbb{R}^{n})},  ~~~0<r<\infty, \\
 &\|\sup _{j \in \mathbb{Z}} 2^{j s} |\dot{\Delta}_{j} f | \|_{L^{p,q}(\mathbb{R}^{n})}, \quad r=\infty. &\ea\right.
$$

\begin{lemma}\label{lem2.5}
Suppose that $f\in\dot{H}^{s}_{p,\infty}(\mathbb{R}^{n})$ with $s\in \mathbb{R}$ and $p\in (1,\infty]$. Then there holds
\be\label{keyinclue}
 \|f\|_{ \dot{B}^{s}_{p,\infty,\infty }}\leq \|f\|_{\dot{F}^{s}_{p,\infty,\infty}}\leq C\|f\|_{\dot{H}^{s}_{p,\infty}(\mathbb{R}^{n})},
\ee
where $C>0$ depends only on $n, s, p$ and $\varphi$.
\end{lemma}
\begin{remark}
It is worth remarking that the nonhomogeneous embedding $H^{s}_{p,\infty}\hookrightarrow F^{s}_{p,\infty,\infty}$ was recently proved by Ko and Lee in \cite{[KL]} and they mentioned that the results hold for the homogeneous case. One can modify the argument in \cite{[KL]} to get the homogeneous case $\dot{H}^{s}_{p,\infty}\hookrightarrow\dot{F}^{s}_{p,\infty,\infty}$. Here, we shall present a different proof via the Hardy-Littlewood maximal function.
\end{remark}
\begin{proof}
Since it is obvious that
$$
\|f\|_{ \dot{B}^{s}_{p,\infty,\infty }}=\sup _{j \in \mathbb{Z}}\| 2^{j s} |\dot{\Delta}_{j} f | \|_{L^{p,\infty}(\mathbb{R}^{n})}\leq \|\sup _{j \in \mathbb{Z}} 2^{j s} |\dot{\Delta}_{j} f | \|_{L^{p,\infty}(\mathbb{R}^{n})}=\|f\|_{\dot{F}^{s}_{p,\infty,\infty}},$$
we focus on the proof
$$
\|f\|_{\dot{F}^{s}_{p,\infty,\infty}}\leq C\|f\|_{\dot{H}^{s}_{p,\infty}(\mathbb{R}^{n})}.
$$
It suffices to show that for all functions $g\in L^{p,\infty}(\mathbb{R}^{n})$,
$$ \|\sup _{j \in \mathbb{Z}}|\phi_{j}\ast g | \|_{L^{p,\infty}(\mathbb{R}^{n})}\leq C\|g\|_{L^{p,\infty}(\mathbb{R}^{n})},$$
where $\phi=\left(|\xi|^{-s}\varphi(\xi)\right)^{\vee}$ and $\phi_{j}(x)=2^{jn}\phi(2^{j}x)$ for all $x\in \mathbb{R}^{n}$.

To this end, observe that $\phi$ is a Schwartz function on $\mathbb{R}^{n}$ and there exists a positive constant $C$ such that for all $x\in \mathbb{R}^{n}$, $|\phi(x)|\leq C (1+|x|)^{-n-1}.$ This yields that for all $j \in \mathbb{Z}$ and $x\in \mathbb{R}^{n}$,
\begin{align*}
|\phi_{j}\ast g(x)|&\leq C\,2^{jn}\int_{\mathbb{R}^{n}}|g(x-y)|(1+|2^{j}y|)^{-n-1}dy
\\& =  C\sum_{k=-\infty }^{\infty}2^{jn}\int_{2^{k}<|y|\leq2^{k+1}}|g(x-y)|(1+|2^{j}y|)^{-n-1}dy\\
&\leq  C\sum_{k=-\infty }^{\infty}2^{jn}(1+2^{j+k})^{-n-1}\int_{2^{k}<|y|\leq2^{k+1}}|g(x-y)|dy\\
&\leq  C\sum_{k=-\infty }^{\infty}2^{(j+k)n}(1+2^{j+k})^{-n-1}\mathcal{M}g(x)\\
&=  C\sum_{k=-\infty }^{\infty}2^{kn}(1+2^{k})^{-n-1}\mathcal{M}g(x)\\
&\leq  C\,\mathcal{M}g(x)(\sum_{k=-\infty }^{-1}2^{kn}+\sum_{k=0}^{\infty}2^{-k})\\
&\leq  C\,\mathcal{M}g(x),
\end{align*}
which implies that there exists a positive constant $C=C(n,\phi)$ such that for all $x\in \mathbb{R}^{n}$,
\be\label{lorentjida}
 \sup _{j \in \mathbb{Z}}|\phi_{j}\ast g(x) | \leq  C\,\mathcal{M}g(x).
 \ee
Then it follows from \eqref{inorm} that
$$ \|\sup _{j \in \mathbb{Z}}|\phi_{j}\ast g | \|_{L^{p,\infty}(\mathbb{R}^{n})}\leq C\|\mathcal{M}g\|_{L^{p,\infty}(\mathbb{R}^{n})} \leq C\|g\|_{L^{p,\infty}(\mathbb{R}^{n})},$$
where $C>0$ depends only on $n, p$ and $\phi$. This concludes the proof.
\end{proof}

\section{Proof of Theorem \ref{the1.1} and \ref{the1.2}}

\subsection{Gagliardo-Nirenberg inequalities in Lorentz   spaces}
  The goal of this subsection is to prove Theorem \ref{the1.1} involving Gagliardo-Nirenberg inequalities in Lorentz spaces.
Firstly, we establish three key inequalities, which play an important role in the proof of three cases in Theorem \ref{the1.1}.
Secondly, in view of a pointwise interpolation estimate for derivatives found in \cite{[BCD]} and equivalent norms of Lorentz spaces, we get
Proposition \ref{keypointwise}. Finally, we are in a position to show Theorem \ref{the1.1}.

\begin{lemma}\label{lem3.1}
\begin{enumerate}[(1)]

\item Suppose that $u \in \dot{H}_{r,p_{1}}^{s}\left(\mathbb{R}^{n}\right)$
    with $1 < r<\infty$, $0 \leq s<n / r$ and $0< p_{1}\leq\infty$. Then there exists a positive constant $C=C(n,s,p,r,p_{1})$ such that
\be\label{3.1}
\|u\|_{L^{p,p_{1}}(\mathbb{R}^{n})} \leq C\|\Lambda^{s}u\|_{L^{r ,p_{1}}(\mathbb{R}^{n})}
~~~\text{with}~~~  \frac{n}{p}=\frac{n}{r}-s.
\ee
\item Let $1 < q,p,r<\infty$, $\,1 \leq l \leq \infty$ and $ s= {n}/{r}.$  Then there exists a positive constant $C $ such that for all functions $u \in L^{p,\infty}\left(\mathbb{R}^{n}\right) \cap \dot{B}^{s}_{r,\infty,\infty}\left(\mathbb{R}^{n}\right)$, there holds
$$\ba
\|u\|_{L^{q,l}(\mathbb{R}^{n})}
 \leq C \|u\|_{L^{p,\infty}(\mathbb{R}^{n})}^{\f{p}{q}}  \| u\|^{1-\f{p}{q}}_{\dot{B}^{s}_{r,\infty,\infty}}~~~\text{with}~~~p<q.
 \ea$$
 \item Let $u \in L^{p,\infty}\left(\mathbb{R}^{n}\right) \cap \dot{B}^{s}_{r,\infty,\infty}\left(\mathbb{R}^{n}\right)$ with $s>n / r  $ and $1 < p, r\leq \infty.$    Then there exists a positive constant $C=C(n,s,p,r)$ such that
$$
\|u\|_{L^{\infty}(\mathbb{R}^{n})} \leq C \|u\|_{L^{p,\infty}(\mathbb{R}^{n})}^{\theta}\|u\|_{\dot{B}^{s}_{r,\infty,\infty}}^{1-\theta},
$$
where
$$
0=\theta \frac{n}{p}+(1-\theta)\left(\frac{n}{r}-s\right),~~0<\theta\leq 1.
$$
 \end{enumerate}
 \end{lemma}
 \begin{remark}
As a corollary of this lemma and imbedding relation in Lemma \ref{lem2.5}, we have
\begin{align}
\label{cripre1}\|\Lambda^{\sigma}u\|_{L^{q,l}(\mathbb{R}^{n})}
 \leq C \|\Lambda^{\sigma}u\|_{L^{p,\infty}(\mathbb{R}^{n})}^{\f{p}{q}}  \| \Lambda^{ \sigma+\f{n}{r}}u\|^{1-\f{p}{q}}_{L^{r,\infty }(\mathbb{R}^{n})}
\end{align}
with $1<p<q<\infty,\,1<r<\infty$ and $1 \leq l \leq \infty$, and
\begin{align}
\label{suppre1}
\|\Lambda^{\sigma}u\|_{L^{\infty}(\mathbb{R}^{n})} \leq C\|\Lambda^{\sigma}u\|_{L^{p,\infty}(\mathbb{R}^{n})}^{\theta}
\|\Lambda^{s}u\|_{L^{r,\infty }(\mathbb{R}^{n})}^{1-\theta}
\end{align}
with $0=\theta n/p+(1-\theta)\left(n/r-s+\sigma\right)$, $0<\theta\leq 1$, $s-\sigma>n/r$ and $1 < p, r\leq \infty$.
 \end{remark}
\begin{remark}
In this lemma, the necessity of the condition that $p,r>1$ results from Lemma \ref{YoungGen}, Young inequality \eqref{young} and Bernstein inequality \eqref{bern2} for Lorentz spaces. Due to \eqref{thirddefincoro}, it is also essential to make the assumption that $q>1$ and $l \geq 1$ to ensure that the space $L^{q,l}(\mathbb{R}^{n})$ is normable.
\end{remark}
\begin{proof}[Proof of  Lemma  \ref{lem3.1}]
(1)
Thanks to Fourier transform, there exists a positive constant $C=C(n,s)$ such that
$$f=\mathcal{F}^{-1}\B(\f{1}{|\xi|^{s}}|\xi|^{s}\hat{f}(\xi) \B)=\mathcal{F}^{-1}\B(\f{1}{|\xi|^{s}}  \B) \ast\Lambda^{s}f
=C\,|\cdot|^{s-n}\ast \Lambda^{s}f.
$$
With the help of the Young inequality \eqref{young} in Lorentz spaces and the fact that $|x|^{-1}\in L^{n,\infty}$, we see that
 $$\ba
\|f \|
_{L^{p,p_{1}}(\mathbb{R}^{n})}
\leq&
C \||\cdot|^{ s-n}\ast \Lambda^{s}f \|
_{L^{p,p_{1}}(\mathbb{R}^{n})}\\
\leq& C \||\cdot|^{ s-n} \|
_{L^{\f{n}{ n-s},\infty}(\mathbb{R}^{n})}\|\Lambda^{s}f \|
_{L^{r ,p_{1}}(\mathbb{R}^{n})}\\
\leq& C \|\Lambda^{s}f \|
_{L^{r ,p_{1}}(\mathbb{R}^{n})}.\ea
$$
(2)
By means of the low and high frequencies, it follows from \eqref{thirddefincoro} that
\be\ba\label{3.4}
\|u\|_{L^{q,l}(\mathbb{R}^{n})}\leq\|u\|_{L^{q,l}(\mathbb{R}^{n})}^{*}
& \leq \|\dot{S}_{j_{0}}u\|_{L^{q,l}(\mathbb{R}^{n})}^{*}+\sum_{j\geq j_{0}}\|\dot{\Delta}_{j }u\|_{L^{q,l}(\mathbb{R}^{n})}^{*}\\
& \leq \frac{q}{q-1}\|\dot{S}_{j_{0}}u\|_{L^{q,l}(\mathbb{R}^{n})}+\frac{q}{q-1}\sum_{j\geq j_{0}}\|\dot{\Delta}_{j }u\|_{L^{q,l}(\mathbb{R}^{n})}. \ea\ee
Here $j_{0}$ is an integer to be chosen later. Observe that
$$\dot{S}_{j_{0}}u= \sum_{k\leq j_{0}-1}\dot{\Delta}_{k}u= h_{j_{0}}\ast u,$$
where $h_{j_{0}}(x)=2^{j_{0}n} h(2^{j_{0}} x)$ for all $x\in\mathbb{R}^{n}$, and
$$h=\mathcal{F}^{-1}\B(\sum_{k\leq -1}\varphi(2^{-k}\xi)\B)\in \mathcal{S}(\mathbb{R}^{n}).$$
Owing to Bernstein's inequality \eqref{bern2} and \eqref{lincreases}, we may apply Lemma \ref{YoungGen} to infer that there exists a positive constant $C$ independent of $j_{0}$ such that
\begin{align*}
\|\dot{S}_{j_{0}}u\|_{L^{q,l}(\mathbb{R}^{n})}\leq & C\,2^{j_{0}n(\f{1}{p}-\f1q)}\|\dot{S}_{j_{0}}u\|_{L^{p,\infty}(\mathbb{R}^{n})}\\
=& C\,2^{j_{0}n(\f{1}{p}-\f1q)}\|h_{j_{0}}\ast u\|_{L^{p,\infty}(\mathbb{R}^{n})}\\
\leq & C\,2^{j_{0}n(\f{1}{p}-\f1q)}\|h_{j_{0}}\|_{L^{1}(\mathbb{R}^{n})}\|u\|_{L^{p,\infty}(\mathbb{R}^{n})}\\
=& C\,2^{j_{0}n(\f{1}{p}-\f1q)}\|h\|_{L^{1}(\mathbb{R}^{n})}\|u\|_{L^{p,\infty}(\mathbb{R}^{n})}.
\end{align*}
For the high-frequency part, it follows from \eqref{Interpolation characteristic}, \eqref{lincreases}, \eqref{bern2} and Lemma \ref{YoungGen} that
\begin{align*}
\sum_{j\geq j_{0}}\|\dot{\Delta}_{j }u\|_{L^{q,l}(\mathbb{R}^{n})}\leq& C\sum_{j\geq j_{0}}\|\dot{\Delta}_{j }u\|_{L^{p,\infty}(\mathbb{R}^{n})}^{1-\alpha} \|\dot{\Delta}_{j }u\|_{L^{q+r,\infty}(\mathbb{R}^{n})}^{\alpha}\\
\leq& C\sum_{j\geq j_{0}}  2^{jn\alpha(\f{1}{r}-\f{1}{q+r})}\|\dot{\Delta}_{j }u\|_{L^{r,\infty}(\mathbb{R}^{n})}^{\alpha}\|\varphi_{j }\ast u\|_{L^{p,\infty}(\mathbb{R}^{n})}^{1-\alpha}\\
\leq& C\sum_{j\geq j_{0}}  2^{-\f{jn\alpha}{q+r}}\|u\|_{\dot{B}^{s}_{r,\infty,\infty}}^{\alpha}\|\varphi_{j }\|_{L^{1}(\mathbb{R}^{n})}^{1-\alpha} \|u\|_{L^{p,\infty}(\mathbb{R}^{n})}^{1-\alpha}\\
=&C\,2^{-\f{j_{0}n\alpha}{q+r}} \|\varphi\|_{L^{1}(\mathbb{R}^{n})}^{1-\alpha} \|u\|_{L^{p,\infty}(\mathbb{R}^{n})}^{1-\alpha}\|u\|_{\dot{B}^{s}_{r,\infty,\infty}}^{\alpha}.
\end{align*}
Here $s=n/r$, $1/q=(1-\alpha)/p+\alpha/(q+r)$ with $0<\alpha<1$, and $\varphi_{j}(x)=2^{jn}\varphi(2^{j}x)$ for all $x\in\mathbb{R}^{n}$. It turns out that
$$
 \ba
\|u\|_{L^{q,l}(\mathbb{R}^{n})}&\leq C\,2^{j_{0}n(\f{1}{p}-\f1q)}\|u\|_{L^{p,\infty}(\mathbb{R}^{n})}
+C\,2^{-\f{j_{0}n\alpha}{q+r}} \|u\|_{L^{p,\infty}(\mathbb{R}^{n})}^{1-\alpha}\|u\|_{\dot{B}^{s}_{r,\infty,\infty}}^{\alpha},
\ea
$$
where the positive constant $C$ is independent of $j_{0}$.
Since $1/p-1/q+\alpha/(q+r)=\alpha/p$, by choosing $j_{0}$ such that $2^{j_{0}n(\f{1}{p}-\f1q)}\|u\|_{L^{p,\infty}
(\mathbb{R}^{n})}\approx 2^{-\f{j_{0}n\alpha}{q+r}} \|u\|_{L^{p,\infty}(\mathbb{R}^{n})}^{1-\alpha}\|u\|_{\dot{B}^{s}_{r,\infty,\infty}}^{\alpha}$,
we may derive that
$$
 \ba
\|u\|_{L^{q,l}(\mathbb{R}^{n})}&\leq   C \|u\|_{L^{p,\infty}(\mathbb{R}^{n})}^{\f{p}{q}}  \| u\|^{1-\f{p}{q}}_{\dot{B}^{s}_{r,\infty,\infty}}.
\ea
$$
(3) If $s>  {n}/{r}$, we take $q= l=\infty$ in \eqref{3.4}.
Using  Bernstein's inequality \eqref{bern1} and Young inequality for Lorentz spaces, we find that
$$\ba
\|\dot{S}_{j_{0}}u\|_{L^{\infty}(\mathbb{R}^{n})}\leq C\,2^{\f{j_{0}n}{p}}\|u\|_{L^{p,\infty}(\mathbb{R}^{n})} ~~\text{and}~~ \sum_{j\geq j_{0}}\|\dot{\Delta}_{j }u\|_{L^{\infty}(\mathbb{R}^{n})}\leq
C\,2^{ {j_{0}n}(\f1r-\f{s}{n})} \| u\|_{\dot{B}^{s}_{r,\infty,\infty}},
\ea$$
where the positive constant $C$ is independent of $j_{0}$. As the derivation of the above, we may choose $j_{0}$ appropriately to conclude that
$$
 \ba
\|u\|_{L^{\infty}(\mathbb{R}^{n})}&\leq C\,2^{\f{j_{0}n }{p}}\|u\|_{L^{p,\infty}(\mathbb{R}^{n})}+C\,2^{ {j_{0}n}(\f1r-\f{s}{n})} \| u\|_{\dot{B}^{s}_{r,\infty,\infty}}\\
&\leq C  \|u\|_{L^{p,\infty}(\mathbb{R}^{n})}^{1-\f{\f{1}{p}}{\f{1}{p}-\f1r+\f{s}{n}}}  \| u\|^{\f{\f{1}{p}}{\f{1}{p}-\f1r+\f{s}{n}}}_{\dot{B}^{s}_{r,\infty,\infty}}.
\ea
$$
This completes the proof of this lemma.
\end{proof}

\begin{prop}\label{keypointwise}
Assume that $u \in L^{q,q_{1}}\left(\mathbb{R}^{n}\right) \cap \dot{H}_{r,r_{1}}^{s}\left(\mathbb{R}^{n}\right) $
with  $1<q, r \leq \infty$ and $1\leq q_{1}, r_{1} \leq \infty$. Then there holds for
 $0 < \sigma<s<\infty $,
$$
\|\Lambda^{\sigma}u\|_{L^{p,p_{1}}(\mathbb{R}^{n})} \leq C\|u\|_{L^{q,q_{1}}(\mathbb{R}^{n})}^{1-\frac{\sigma}{s}}
\|\Lambda^{s}u\|_{L^{r,r_{1}}(\mathbb{R}^{n})}^{\frac{\sigma}{s}},
$$
with
$$
\frac{1}{p}=\left(1-\frac{\sigma}{s}\right) \frac{1}{q}+\frac{ \sigma}{sr},~~\frac{1}{p_{1}}=\left(1-\frac{\sigma}{s}\right) \frac{1}{q_{1}}+  \frac{\sigma }{s r_{1}}.
$$
Here $C$ is a positive constant depending only on $q,r,q_{1},r_{1},s,\sigma$ and $n$.
\end{prop}
\begin{remark}As a special case of this proposition, there holds the following estimate
$$
\|\Lambda^{\sigma}u\|_{L^{p,\infty}(\mathbb{R}^{n})} \leq C\|u\|_{L^{q,\infty}(\mathbb{R}^{n})}^{1-\frac{\sigma}{s}}
\|\Lambda^{s}u\|_{L^{r,\infty}(\mathbb{R}^{n})}^{\frac{\sigma}{s}},
$$
where $q,r,p,s$ and $\sigma$ satisfy the same conditions as in Proposition \ref{keypointwise}. This inequality will be used in the proof of Theorem \ref{the1.1}.
\end{remark}
\begin{remark}
Very recently, by means of  Maz'ya-Shaposhnikova pointwise estimates in \cite{[MS]}, Fiorenza-Formica-Rosaria-Soudsky  \cite{[FFRS]} showed  the following
estimate
$$
\left\|\nabla^{j} u\right\|_{L^{p, p_{1}}(\mathbb{R}^{n})} \leq C\left\|\nabla^{k} u\right\|_{L^{r, r_{1}}(\mathbb{R}^{n})}^{\frac{j}{k}}\|u\|_{L^{q, q_{1}}(\mathbb{R}^{n})}^{1-\frac{j}{k}},
$$
where $j, k \in \mathbb{N}, 1 \leq j<k, 1<q, r\leq \infty,$ $1\leq q_{1}, r_{1 }\leq \infty$ and
$$
\frac{1}{p}=\frac{\frac{j}{k}}{r}+\frac{1-\frac{j}{k}}{q}, \quad \frac{1}{p_{1}}=\frac{\frac{j}{k}}{r_{1}}+\frac{1-\frac{j}{k}}{q_{1}}.
$$
Proposition  \ref{keypointwise} extends the aforementioned integer cases of Gagliardo-Nirenberg inequalities in \cite{[FFRS]} to the fractional cases.
\end{remark}
\begin{proof}
Thanks to the decomposition of low and high frequencies, one can derive the following pointwise estimate
\be\label{bcd}
 \left|\Lambda^{\sigma} u(x)\right| \leq C(\mathcal{M} u(x))^{1-\frac{\sigma}{s}}\left(\mathcal{M} \Lambda^{s} u(x)\right) ^{\frac{\sigma}{s}},
\ee
whose proof can be found in \cite[p.84]{[BCD]}.

As a consequence, it follows from the H\"older's inequality for Lorentz spaces that
$$\ba
 \|\Lambda^{\sigma} u \|_{L^{p,p_{1}}(\mathbb{R}^{n})}&\leq C \|(\mathcal{M} u )^{1-\frac{\sigma}{s}}\left(\mathcal{M} \Lambda^{s} u \right) ^{\frac{\sigma}{s}}\|_{L^{p,p_{1}}(\mathbb{R}^{n})}\\& \leq C \|( \mathcal{M} u )^{1-\f{\sigma}{s}}\|_{L^{\f{sq}{s-\sigma}, \f{sq_{1}}{s-\sigma}}(\mathbb{R}^{n})}\|\left( \mathcal{M} \Lambda^{s} u \right)^{\f{\sigma}{s}}\|_{L^{\f{ sr}{\sigma}, \f{s r_{1}}{\sigma}}(\mathbb{R}^{n})}\\& \leq C \|\mathcal{M} u \|^{1-\f{\sigma}{s}}_{L^{q,  q_{1}  }(\mathbb{R}^{n})}\|\mathcal{M}\left(  \Lambda^{s } u\right)\|^{\f{\sigma}{s}}_{L^{r, r_{1} }(\mathbb{R}^{n})}.
 \ea$$
According to \eqref{inorm}, we conclude the desired estimate.
\end{proof}

Now, at this stage, we can prove Theorem \ref{the1.1}.

\begin{proof}[Proof of Theorem \ref{the1.1}]
($I$) First, we consider \eqref{GNL} under the hypothesis that $0< s-\sigma<\frac{n}{r}$.

($I_{1}$) If $\sigma=0$, it follows from the interpolation characteristic \eqref{Interpolation characteristic} of Lorentz spaces that
\be\label{gi}
\|u\|_{L^{p, 1}\left(\mathbb{R}^{n}\right)} \leq C\|u\|_{L^{\tilde{p}, \infty}\left(\mathbb{R}^{n}\right)}^{1-\theta}\|u\|_{L^{q, \infty}\left(\mathbb{R}^{n}\right)}^{\theta}, ~~\text{with}~~\f{1}{p }=\f{1-\theta}{\tilde{p}}+\f{\theta}{q},~~0<\theta<1,
 \ee
where we have used the fact that $\,\f{1}{\tilde{p}}=\f{1}{r}-\f{s}{n}\neq \f{1}{q}$.
This together with the Sobolev inequality \eqref{3.1} ensures that
$$
\|u\|_{L^{p, 1}\left(\mathbb{R}^{n}\right)} \leq C\|\Lambda^{s}u \|
_{L^{r ,\infty}(\mathbb{R}^{n})}^{1-\theta}\|u\|_{L^{q, \infty}\left(\mathbb{R}^{n}\right)}^{\theta}, $$
with
$$\f{1}{p }= (1-\theta)(\f{1}{r}-\f{s}{n})+\f{\theta}{q},~~0<\theta<1.$$

($I_{2}$) If $\sigma>0$,
the Sobolev  embedding \eqref{3.1} yields
\be\label{sub1}
\left\|\Lambda^{\sigma} u\right\|_{L^{r^{*},\infty}(\mathbb{R}^{n})} \leq C\left\|\Lambda^{s} u\right\|_{L^{r,\infty}(\mathbb{R}^{n})},
\ee
with
$$
\frac{1}{r^{*}}=\frac{1}{r}-\frac{s-\sigma}{n}.
$$
It follows from Proposition \ref{keypointwise} that for $1<q, r \leq \infty$,
\be\label{sub3}
\left\|\Lambda^{\sigma} u\right\|_{L^{\tilde{p},\infty}(\mathbb{R}^{n})} \leq C\|u\|_{L^{q,\infty}(\mathbb{R}^{n})}^{1-\frac{\sigma}{s}}\left\|\Lambda^{s} u\right\|_{L^{r,\infty}(\mathbb{R}^{n})}^{\frac{\sigma}{s}}
\ee
with
$$
\frac{1}{\widetilde{p}}=\left(1-\frac{\sigma}{s}\right) \frac{1}{q}+ \frac{ \sigma}{s r}.
$$
Then the hypothesis on the indices that
$$\frac{n}{p}-\sigma=\theta \frac{n}{q}+(1-\theta)\left(\frac{n}{r}-s\right),~~0 <\theta < 1-\frac{\sigma}{s},
$$
imply the following relation
$$
\frac{1}{p}= \frac{\alpha }{\widetilde{p}}+\frac{1-\alpha}{r^{*}},~~\alpha=\frac{\theta}{1-\frac{\sigma}{s}}\in (0,1).
$$
Observe that the condition $\,s-\f{n}{r} \neq \sigma-\frac{n}{p}\,$ guarantees $\,\widetilde{p}\neq r^{*}$. From the interpolation characteristic \eqref{Interpolation characteristic} of Lorentz spaces, we see that
\be\label{sub2}
 \|\Lambda^{\sigma} u \|_{L^{p,1}(\mathbb{R}^{n})} \leq C\|\Lambda^{\sigma} u \|_{L^{\tilde{p},\infty}(\mathbb{R}^{n})}^{\alpha} \|\Lambda^{\sigma} u \|_{L^{r^{\ast},\infty}(\mathbb{R}^{n})}^{1-\alpha}.
\ee
Plugging \eqref{sub1} and \eqref{sub3} into \eqref{sub2}, we get
$$
\begin{aligned}
\left\|\Lambda^{\sigma} u\right\|_{L^{p,1}(\mathbb{R}^{n})} & \leq C\left(\|u\|_{L^{q,\infty}(\mathbb{R}^{n})}^{1-\frac{\sigma}{s}}\left\|\Lambda^{s} u\right\|_{L^{r,\infty}(\mathbb{R}^{n})}^{\frac{\sigma}{s}}\right)^{\alpha}\left\|\Lambda^{s} u\right\|_{L^{r,\infty}(\mathbb{R}^{n})}^{1-\alpha} \\
&=C\|u\|_{L^{q,\infty}(\mathbb{R}^{n})}^{\left(1-\frac{\sigma}{s}\right) \alpha}\left\|\Lambda^{s} u\right\|_{L^{r,\infty}(\mathbb{R}^{n})}^{1-\alpha\left(1-\frac{\sigma}{s}\right)}\\
&=C\|u\|_{L^{q,\infty}(\mathbb{R}^{n})}^{\theta}\left\|\Lambda^{s} u\right\|_{L^{r,\infty}(\mathbb{R}^{n})}^{1-\theta}.
\end{aligned}
$$

($II$) Next, we turn our attention to the case that $s-\sigma= {n}/{r}$ in \eqref{GNL}. By virtue of  \eqref{cripre1}, we have proved  \eqref{GNL} with $s= {n}/{r}$ and $\sigma=0$. In the following, we assume that $\sigma>0$.
According to  Proposition \ref{keypointwise}, we arrive at that for $\frac{1}{\tilde{p}}=\left(1-\frac{\sigma}{s}\right) \frac{1}{q}+\frac{\sigma}{sr}=\left(1-\frac{\sigma}{s}\right)\left(\frac{1}{q}+\frac{\sigma}{n}\right)$,
\be\label{crica1}
\|\Lambda^{\sigma} u \|_{L^{\tilde{p},\infty}(\mathbb{R}^{n})}\leq
C \|u\|_{L^{q,\infty}(\mathbb{R}^{n})}^{1-\frac{\sigma}{s}}\left\|\Lambda^{s} u\right\|_{L^{r,\infty}(\mathbb{R}^{n})}^{\frac{\sigma}{s}}.
\ee
Since $\frac{1}{p}=\theta\left(\frac{1}{q}+\frac{\sigma}{n}\right)<\left(1-\frac{\sigma}{s}\right)\left(\frac{1}{q}+\frac{\sigma}{n}\right)=\frac{1}{\tilde{p}}$, it follows from \eqref{cripre1} that
\be\label{cri1}
\left\|\Lambda^{\sigma} u\right\|_{L^{p,1}(\mathbb{R}^{n})}   \leq C\left\|\Lambda^{\sigma} u\right\|_{L^{\tilde{p},\infty}(\mathbb{R}^{n})}^{\frac{\tilde{p}}{p}}\left\|\Lambda^{s} u\right\|_{L^{r,\infty}(\mathbb{R}^{n})}^{1-\frac{\tilde{p}}{p}}.
\ee
Substituting \eqref{crica1} into \eqref{cri1}, we obtain
$$
\begin{aligned}
\left\|\Lambda^{\sigma} u\right\|_{L^{p,1}(\mathbb{R}^{n})}\leq C \|u\|_{L^{q,\infty}(\mathbb{R}^{n})}^{\frac{\tilde{p}}{p}\left(1-\frac{\sigma}{s}\right)}\left\|\Lambda^{s} u\right\|_{L^{r,\infty}(\mathbb{R}^{n})}^{1-\frac{\tilde{p}}{p}\left(1-\frac{\sigma}{s}\right)}=C \|u\|_{L^{q,\infty}(\mathbb{R}^{n})}^{\theta}\left\|\Lambda^{s} u\right\|_{L^{r,\infty}(\mathbb{R}^{n})}^{1-\theta}.
\end{aligned}
$$

($III$) Finally, it is enough to show    \eqref{GNL} under the
hypothesis that $s-\sigma> {n}/{r}$.

($III_{1}$) If $\sigma>0$, we conclude from Proposition \ref{keypointwise} that for $1<q, r \leq \infty$,
\be\label{subb11}
\left\|\Lambda^{\sigma} u\right\|_{L^{\tilde{p},\infty}(\mathbb{R}^{n})} \leq C\|u\|_{L^{q,\infty}(\mathbb{R}^{n})}^{1-\frac{\sigma}{s}}\left\|\Lambda^{s} u\right\|_{L^{r,\infty}(\mathbb{R}^{n})}^{\frac{\sigma}{s}},
\ee
where
$$
\frac{1}{\widetilde{p}}=\left(1-\frac{\sigma}{s}\right) \frac{1}{q}+ \frac{ \sigma}{s r}.
$$
Observe that
\begin{align*}
\f{1}{p }=\theta \left(\f{1}{q}+\f{s}{n}-\f{1}{r}\right)+\left(\f{1}{r}-\f{s}{n}+\f{\sigma}{n}\right)
<\left(1-\frac{\sigma}{s}\right)\left(\f{1}{q}+\f{s}{n}-\f{1}{r}\right)+\left(\f{1}{r}-\f{s}{n}+\f{\sigma}{n}\right)=\frac{1}{\widetilde{p}}\,.
\end{align*}
In view of  the interpolation characteristic \eqref{Interpolation characteristic} of Lorentz spaces, we see that
\be\label{sup2}
\left\|\Lambda^{\sigma} u\right\|_{L^{p,1}(\mathbb{R}^{n})} \leq C\left\|\Lambda^{\sigma} u\right\|_{L^{\tilde{p},\infty}(\mathbb{R}^{n})}^{\alpha}\left\|\Lambda^{\sigma} u\right\|_{L^{\infty}(\mathbb{R}^{n})}^{1-\alpha}
\ee
with
$$
\frac{1}{p}=\frac{\alpha}{\widetilde{p}}
+\frac{1-\alpha}{\infty},~~0<\alpha< 1.
$$
Furthermore \eqref{suppre1} ensures that
\be\label{sup1}
\left\|\Lambda^{\sigma} u\right\|_{L^{\infty}(\mathbb{R}^{n})}\leq C\left\|\Lambda^{\sigma} u\right\|_{L^{\tilde{p},\infty}(\mathbb{R}^{n})}^{\beta}\left\|\Lambda^{s} u\right\|_{L^{r,\infty}(\mathbb{R}^{n})}^{1-\beta},
\ee
where
$$0= \frac{\beta n}{\tilde{p}}+(1-\beta)\left(\frac{n}{r}-s+\sigma\right),~~0<\beta\leq 1.$$
Inserting \eqref{sup1} into \eqref{sup2} and using \eqref{subb11}, we have
\begin{align*}
\left\|\Lambda^{\sigma} u\right\|_{L^{p,1}(\mathbb{R}^{n})} \leq
& C\|u\|_{L^{q,\infty}(\mathbb{R}^{n})}^{\left(1-\frac{\sigma}{s}\right)[\alpha+(1-\alpha) \beta]}\left\|\Lambda^{s} u\right\|_{L^{r,\infty}(\mathbb{R}^{n})}^{(1-\beta)(1-\alpha)+\frac{\sigma}{s}[\alpha+(1-\alpha) \beta]}\\
=& C\|u\|_{L^{q,\infty}(\mathbb{R}^{n})}^{\theta}\left\|\Lambda^{s} u\right\|_{L^{r,\infty}(\mathbb{R}^{n})}^{1-\theta}.
\end{align*}

($III_{2}$) If $\sigma=0$, we note that
\begin{align*}
\f{1}{p}= \f{\theta}{q}+(1-\theta)\left(\f{1}{r}-\f{s}{n}\right)<\frac{1}{q}\,.
\end{align*}
Then it follows from the interpolation characteristic \eqref{Interpolation characteristic} of Lorentz spaces that
\be\label{giWA}
\|u\|_{L^{p, 1}\left(\mathbb{R}^{n}\right)} \leq C\|u\|_{L^{q, \infty}\left(\mathbb{R}^{n}\right)}^{\tau}\|u\|_{L^{\infty}\left(\mathbb{R}^{n}\right)}^{1-\tau}, ~~\text{with}~~\f{1}{p }=\f{\tau}{q}+\f{1-\tau}{\infty}~~\text{and}~~0<\tau<1.
 \ee
In view of \eqref{suppre1}, we obtain
$$
\left\| u\right\|_{L^{\infty}(\mathbb{R}^{n})}\leq C\left\| u\right\|_{L^{q,\infty}(\mathbb{R}^{n})}^{\lambda}\left\|\Lambda^{s} u\right\|_{L^{r,\infty}(\mathbb{R}^{n})}^{1-\lambda},~~\text{with}~~0=\f{\lambda}{q}+(1-\lambda)\left(\frac{1}{r}-\frac{s}{n}\right)~~\text{and}~~0<\lambda\leq 1.
$$
This together with \eqref{giWA} yields that
\begin{align*}
\|u\|_{L^{p, 1}\left(\mathbb{R}^{n}\right)} \leq C\|u\|_{L^{q, \infty}\left(\mathbb{R}^{n}\right)}^{\tau+\lambda(1-\tau)}\left\|\Lambda^{s} u\right\|_{L^{r,\infty}(\mathbb{R}^{n})}^{(1-\lambda)(1-\tau)}
= C\|u\|_{L^{q,\infty}(\mathbb{R}^{n})}^{\theta}\left\|\Lambda^{s} u\right\|_{L^{r,\infty}(\mathbb{R}^{n})}^{1-\theta}.
\end{align*}
We complete the proof of Theorem  \ref{the1.1}.
 \end{proof}
 \subsection{Gagliardo-Nirenberg inequalities in Besov-Lorentz spaces}
 In this subsection, by means of the Littlewood-Paley decomposition and generalized Bernstein inequalities in Lemma \ref{lem2.4}, we shall prove the Gagliardo-Nirenberg inequality  \eqref{glgn} in Besov-Lorentz spaces.
  \begin{proof}[Proof of Theorem \ref{the1.2}]
First, we assert  that
$q\leq p$, $\sigma>0$ and $\f{n}{p}-\sigma=\theta\f{n}{q}+(1-\theta)(\f{n}{r}-s)$ imply that
\be\label{first}
s-\sigma-\f{n}{r}+\f{n}{p}>0,
\ee
which will be frequently used later.
Indeed, thanks to $q\leq p$ and $\sigma>0$, we see that
$$
\f{n}{p}-(1-\theta)\sigma>\f{n}{p}-\sigma
=\theta\f{n}{q}+(1-\theta)(\f{n}{r}-s)\geq\theta\f{n}{p}
+(1-\theta)(\f{n}{r}-s),
$$
 that is,
$$
s-\sigma-\f{n}{r}+\f{n}{p}>0.
$$
The assertion follows. Similarly, we also assert that $q< p$ and $\sigma=0$ yield \eqref{first}.

($I$) We next consider the case  $q=r$ of \eqref{glgn}. Note that
$$
\f1p=\f{\theta}{q}+\f{1-\theta}{r}-\f{1}{n}[(1-\theta)s-\sigma]
<\f{\theta}{q}+\f{1-\theta}{r}=\f1q.
$$
Therefore, we see that $r=q<p$ and \eqref{first} are valid.
With the help of the Bernstein inequality \eqref{bern2} in Lorentz spaces, we infer that
\be\ba\label{5.2}
\|u\|_{\dot{B}^{\sigma}_{p,1,1}}&=\sum_{j\leq k}2^{j\sigma}\|\dot{\Delta}_{j}u\|_{L^{p,1}(\mathbb{R}^{n})}+\sum_{j> k}2^{j\sigma}\|\dot{\Delta}_{j}u\|_{L^{p,1}(\mathbb{R}^{n})}\\
&\leq C\sum_{j\leq k}2^{j[\sigma+n(\f1r-\f1p)]}\|\dot{\Delta}_{j}u\|_{L^{r,\infty}(\mathbb{R}^{n})}+C\sum_{j> k}2^{-j[s-\sigma-n(\f1r-\f1p)]}2^{js}\|\dot{\Delta}_{j}u\|_{L^{r,\infty}(\mathbb{R}^{n})}\\
&\leq C\f{2^{k[\sigma+n(\f1r-\f1p)]}}{1-2^{ -[\sigma+n(\f1r-\f1p)]}} \| u\|_{\dot{B}^{0}_{r,\infty,\infty}}+C\f{ 2^{-k[s-\sigma-n(\f1r-\f1p)]}}{1-2^{- [s-\sigma-n(\f1r-\f1p)]}}  \| u\|_{\dot{B}^{s}_{r,\infty,\infty}},
\ea\ee
where we have used $\sigma+n(\f1r-\f1p)>0$ and \eqref{first}.

As a consequence, by choosing the integer $k$ appropriately such that $ 2^{k[\sigma+n(\f1r-\f1p)]}  \| u\|_{\dot{B}^{0}_{r,\infty,\infty}}\approx   2^{-k[s-\sigma-n(\f1r-\f1p)]} \| u\|_{\dot{B}^{s}_{r,\infty,\infty}}$, we further get
$$\ba
\|u\|_{\dot{B}^{\sigma}_{p,1,1}} &\leq C  \| u\|_{\dot{B}^{0}_{r,\infty,\infty}} ^{\theta} \| u\|_{\dot{B}^{s}_{r,\infty,\infty}}^{1-\theta}.
\ea$$

($II$) We turn our attention to the  case $q<r$. In order to get
 \eqref{first}, we check that $q<p$  via the following straightforward  calculation
$$\ba
\f{n}{p}&=(1-\f{\sigma}{s})\f{n}{q}+\f{\sigma}{s}\f{n}{r}+
(1-\f{\sigma}{s}-\theta)(\f{n}{r}-\f{n}{q}-s)\\
&< (1-\f{\sigma}{s})\f{n}{q}+\f{\sigma}{s}\f{n}{q}
\\&=\f{n}{q}.
\ea$$
We also see that
\be\label{second}
\f{ 1}{p}
 < (1-\f{\sigma}{s})\f{1}{q}+\f{\sigma}{s}\f{1}{r}.
\ee
The following discussion will be divided into three subcases that $q<r<p$, $ q<p< r$ and   $ q<p= r$.

($II_{1}$) We examine the case $q<r<p$.
As the derivation of \eqref{5.2}, we find that
$$\ba
\|u\|_{\dot{B}^{\sigma}_{p,1,1}}
&\leq C\sum_{j\leq k}2^{j[\sigma+n(\f1q-\f1p)]}\|\dot{\Delta}_{j}u\|_{L^{q,\infty}}+C\sum_{j> k}2^{-j[s-\sigma-n(\f1r-\f1p)]}2^{js}\|\dot{\Delta}_{j}u\|_{L^{r,\infty}}\\
&\leq C\f{2^{k[\sigma+n(\f1q-\f1p)]}}{1-2^{ -[\sigma+n(\f1q-\f1p)]}} \| u\|_{\dot{B}^{0}_{q,\infty,\infty}}+C\f{ 2^{-k[s-\sigma-n(\f1r-\f1p)]}}{1-2^{- [s-\sigma-n(\f1r-\f1p)]}}  \| u\|_{\dot{B}^{s}_{r,\infty,\infty}},
\ea$$
where we have used $\sigma+n(\f1q-\f1p)>0$ and \eqref{first}.

Hence, we conclude that
$$\ba
\|u\|_{\dot{B}^{\sigma}_{p,1,1}} &\leq C  \| u\|_{\dot{B}^{0}_{q,\infty,\infty}} ^{\theta} \| u\|_{\dot{B}^{s}_{r,\infty,\infty}}^{1-\theta}.
\ea$$

($II_{2}$) We deal with the case $ q<p< r$.
By means of  the interpolation characteristic \eqref{Interpolation characteristic} of Lorentz spaces, we
observe that
$$
\|\dot{\Delta}_{j}u\|_{L^{p,1}(\mathbb{R}^{n})}\leq C
\|\dot{\Delta}_{j}u\|^{\alpha}_{L^{q,\infty}(\mathbb{R}^{n})}
\|\dot{\Delta}_{j}u\|^{1-\alpha}_{L^{r,\infty}(\mathbb{R}^{n})},
~~\f{1}{p}=\f{\alpha}{q}+\f{1-\alpha}{r}.
$$
Combining this, the Bernstein inequality \eqref{bern2} in Lorentz spaces and \eqref{first}, we know that
\be\label{3.17}\ba
\|u\|_{\dot{B}^{\sigma}_{p,1,1}}&=\sum_{j\leq k}2^{j\sigma}\|\dot{\Delta}_{j}u\|_{L^{p,1}(\mathbb{R}^{n})}+\sum_{j> k}2^{j\sigma}\|\dot{\Delta}_{j}u\|_{L^{p,1}(\mathbb{R}^{n})}\\
&\leq C\sum_{j\leq k}2^{j[\sigma+n(\f1q-\f1p)]}\|\dot{\Delta}_{j}u\|_{L^{q,\infty}(\mathbb{R}^{n})}+C\sum_{j> k}2^{j\sigma}\|\dot{\Delta}_{j}u\|^{\alpha}_{L^{q,\infty}(\mathbb{R}^{n})}
\|\dot{\Delta}_{j}u\|^{1-\alpha}_{L^{r,\infty}(\mathbb{R}^{n})}\\
&\leq C\f{2^{k[\sigma+n(\f1q-\f1p)]}}{1-2^{- [\sigma+n(\f1q-\f1p)]}} \| u\|_{\dot{B}^{0}_{q,\infty,\infty}}+C\f{ 2^{-k[s(1-\alpha)-\sigma ]}}{1-2^{- [s(1-\alpha)-\sigma ]}} \| u\|_{\dot{B}^{0}_{q,\infty,\infty}}^{\alpha} \| u\|_{\dot{B}^{s}_{r,\infty,\infty}}^{1-\alpha},
\ea\ee
where we have used $\sigma+n(\f1q-\f1p)>0$ and $s(1-\alpha)-\sigma>0$ which is derived
from \eqref{second}. Choosing the integer $k$ such that $ 2^{k[\sigma+n(\f1q-\f1p)]}\| u\|_{\dot{B}^{0}_{q,\infty,\infty}}\approx
2^{-k[s(1-\alpha)-\sigma ]}\| u\|_{\dot{B}^{0}_{q,\infty,\infty}}^{\alpha} \| u\|_{\dot{B}^{s}_{r,\infty,\infty}}^{1-\alpha}$,  we also have
$$\ba
\|u\|_{\dot{B}^{\sigma}_{p,1,1}} &\leq C  \| u\|_{\dot{B}^{0}_{q,\infty,\infty}} ^{\theta} \| u\|_{\dot{B}^{s}_{r,\infty,\infty}}^{1-\theta}.
\ea$$

($II_{3}$) We treat the case $ q<p= r$.
It follows from the interpolation characteristic \eqref{Interpolation characteristic} of Lorentz spaces that
\be\label{5191}
\|\dot{\Delta}_{j}u\|_{L^{p,1}(\mathbb{R}^{n})}\leq C
\|\dot{\Delta}_{j}u\|^{1-\alpha}_{L^{q,\infty}(\mathbb{R}^{n})}
\|\dot{\Delta}_{j}u\|^{\alpha}_{L^{(1+\varepsilon)p,\infty}(\mathbb{R}^{n})},
~~\f{1}{p}=\f{1-\alpha}{q}+\f{\alpha}{(1+\varepsilon)p},
\ee
where $\varepsilon>0$ will be determined later.
We derive from  this,  the Bernstein inequality \eqref{bern2} in Lorentz spaces and \eqref{first} that
\be\label{3.171}\ba
\|u\|_{\dot{B}^{\sigma}_{p,1,1}}&=\sum_{j\leq k}2^{j\sigma}\|\dot{\Delta}_{j}u\|_{L^{p,1}(\mathbb{R}^{n})}+\sum_{j> k}2^{j\sigma}\|\dot{\Delta}_{j}u\|_{L^{p,1}(\mathbb{R}^{n})}\\
&\leq C\sum_{j\leq k}2^{j[\sigma+n(\f1q-\f1p)]}\|\dot{\Delta}_{j}u\|_{L^{q,\infty}(\mathbb{R}^{n})}+C\sum_{j> k}2^{j\sigma}\|\dot{\Delta}_{j}u\|^{1-\alpha}_{L^{q,\infty}(\mathbb{R}^{n})}
\|\dot{\Delta}_{j}u\|^{\alpha}_{L^{(1+\varepsilon)p,\infty}(\mathbb{R}^{n})}\\
&\leq C\f{2^{k[\sigma+n(\f1q-\f1p)]}}{1-2^{ -[\sigma+n(\f1q-\f1p)]}} \| u\|_{\dot{B}^{0}_{q,\infty,\infty}}+ C\sum_{j> k}2^{-j[s\alpha- \f{n\varepsilon\alpha}{(1+\varepsilon)p}-\sigma]} \| u\|_{\dot{B}^{0}_{q,\infty,\infty}}^{1-\alpha} \| u\|_{\dot{B}^{s}_{r,\infty,\infty}}^{\alpha},
\ea\ee
where we have used the fact that $\sigma+n(\f1q-\f1p)>0$. \\
Denote $\delta(\varepsilon)=s\alpha- \f{n\varepsilon\alpha}{(1+\varepsilon)p}-\sigma$. From \eqref{5191}, we see that
 $$\delta(\varepsilon)=\f{ps(1+\varepsilon)(p-q)-\sigma p[(1+\varepsilon)p-q]-\varepsilon n(p-q)}{p[(1+\varepsilon)p-q]},$$
 and $\delta(\varepsilon)$ is a continuous function on a neighborhood of $0$. Since $\delta(0)>0$, there exists a sufficiently small $\varepsilon>0$ such that $\delta(\varepsilon)>0$.
 It follows from \eqref{3.171} that
 \be\ba
\|u\|_{\dot{B}^{\sigma}_{p,1,1}}\leq C\f{2^{k[\sigma+n(\f1q-\f1p)]}}{1-2^{- [\sigma+n(\f1q-\f1p)]}} \| u\|_{\dot{B}^{0}_{q,\infty,\infty}}+ C\f{2^{-k[s\alpha- \f{n\varepsilon\alpha}{(1+\varepsilon)p}-\sigma]}}{1-2^{-[s\alpha- \f{n\varepsilon\alpha}{(1+\varepsilon)p}-\sigma]}}  \| u\|_{\dot{B}^{0}_{q,\infty,\infty}}^{1-\alpha} \| u\|_{\dot{B}^{s}_{r,\infty,\infty}}^{\alpha},
\ea\ee
 which also yields that
$$\ba
\|u\|_{\dot{B}^{\sigma}_{p,1,1}} &\leq C  \| u\|_{\dot{B}^{0}_{q,\infty,\infty}} ^{\theta} \| u\|_{\dot{B}^{s}_{r,\infty,\infty}}^{1-\theta}.
\ea$$

($III$) Finally, it remains to show  \eqref{glgn} under the case that $q>r$.

($III_{1}$) We  first consider \eqref{glgn} under  the  hypothesis that $r<q\leq p$.  We divide this case into two subcases that $r<q <p$ and $r<q =p$.

($III_{11}$)  We handle with the   case  $r<q <p$.
In view of the Bernstein inequality \eqref{bern2}, we see that
$$\ba
\|u\|_{\dot{B}^{\sigma}_{p,1,1}}
&\leq C\sum_{j\leq k}2^{j[\sigma+n(\f1q-\f1p)]}\|\dot{\Delta}_{j}u\|_{L^{q,\infty}(\mathbb{R}^{n})}+C\sum_{j> k}2^{-j[s-\sigma-n(\f1r-\f1p)]}2^{js}\|\dot{\Delta}_{j}u\|_{L^{r,\infty}(\mathbb{R}^{n})}\\
&\leq C\f{2^{k[\sigma+n(\f1q-\f1p)]}}{1-2^{ -[\sigma+n(\f1q-\f1p)]}} \| u\|_{\dot{B}^{0}_{q,\infty,\infty}}+C\f{ 2^{-k[s-\sigma-n(\f1r-\f1p)]}}{1-2^{- [s-\sigma-n(\f1r-\f1p)]}}  \| u\|_{\dot{B}^{s}_{r,\infty,\infty}},
\ea$$
where we have used $\sigma+n(\f1q-\f1p)>0$ and \eqref{first}.

Therefore, we conclude that
$$\ba
\|u\|_{\dot{B}^{\sigma}_{p,1,1}} &\leq C  \| u\|_{\dot{B}^{0}_{q,\infty,\infty}} ^{\theta} \| u\|_{\dot{B}^{s}_{r,\infty,\infty}}^{1-\theta}.
\ea$$

($III_{12}$) We need to show  \eqref{glgn} under  the  hypothesis  that $r<q=p$. Observe that the condition $\,s-\f{n}{r} \neq \sigma-\frac{n}{p}\,$ implies that $\sigma>0 $ in this case. In the same manner as   \eqref{5191}, we see that
\be
\|\dot{\Delta}_{j}u\|_{L^{p,1}(\mathbb{R}^{n})}\leq C
\|\dot{\Delta}_{j}u\|^{1-\alpha}_{L^{r,\infty}(\mathbb{R}^{n})}
\|\dot{\Delta}_{j}u\|^{\alpha}_{L^{(1+\varepsilon)p,\infty}(\mathbb{R}^{n})},
~~\f{1}{p}=\f{1-\alpha}{r}+\f{\alpha}{(1+\varepsilon)p},
\ee
where $\varepsilon>0$ will be determined later. Then we obtain
$$\ba
\|u\|_{\dot{B}^{\sigma}_{p,1,1}}
&\leq C\sum_{j\leq k}2^{j\sigma}\|\dot{\Delta}_{j}u\|^{1-\alpha}_{L^{r,\infty}(\mathbb{R}^{n})}
\|\dot{\Delta}_{j}u\|^{\alpha}_{L^{(1+\varepsilon)p,\infty}(\mathbb{R}^{n})}+C\sum_{j> k}2^{-j[s-\sigma-n(\f1r-\f1p)]}2^{js}\|\dot{\Delta}_{j}u\|_{L^{r,\infty}(\mathbb{R}^{n})}\\
&\leq  C\sum_{j\leq k}2^{j[\sigma+\f{n\varepsilon\alpha}{p(1+\varepsilon)}-s(1-\alpha)]}  \| u\|_{\dot{B}^{0}_{q,\infty,\infty}}^{\alpha} \| u\|_{\dot{B}^{s}_{r,\infty,\infty}}^{1-\alpha}+C\f{ 2^{-k[s-\sigma-n(\f1r-\f1p)]}}{1-2^{- [s-\sigma-n(\f1r-\f1p)]}}  \| u\|_{\dot{B}^{s}_{r,\infty,\infty}},
\ea$$
As the arguments in ($II_{3}$), we can choose $\varepsilon>0$ sufficiently small to ensure that
$\sigma+\f{n\varepsilon\alpha}{p(1+\varepsilon)}-s(1-\alpha)>0$. This yields the desired inequality \eqref{glgn}. We omit the details.

($III_{2}$) Let $r<q$ and $ p<q$.
It is clear that
\be\label{wwkey}
\f1p=\B(\f{1-\f{\sigma}{s}-\theta}{1-\f{\sigma}{s}}\B)
\B(\f1r-\f{s-\sigma}{n}\B)+
\B(\f{ \theta}{1-\f{\sigma}{s}}\B)\B((1-\f{\sigma}{s})\f1q+\f{\sigma}{sr}\B),
\ee
which yields that $r<p$ in this case. Additionally, the condition $\,s-\f{n}{r} \neq \sigma-\frac{n}{p}\,$ guarantees that $\f{ 1}{p}
 \neq(1-\f{\sigma}{s})\f{1}{q}+\f{\sigma}{s}\f{1}{r}$.

($III_{21}$) We assume that $r<q,$ $p<q $ and $\f{ 1}{p}
 < (1-\f{\sigma}{s})\f{1}{q}+\f{\sigma}{s}\f{1}{r}$. Then \eqref{first} follows from \eqref{wwkey}.
Note that $r<p<q$, a slight modification of the proof of \eqref{3.17} together with \eqref{first} implies that
$$\ba
\|u\|_{\dot{B}^{\sigma}_{p,1,1}}&=\sum_{j\leq k}2^{j\sigma}\|\dot{\Delta}_{j}u\|_{L^{p,1}(\mathbb{R}^{n})}+\sum_{j> k}2^{j\sigma}\|\dot{\Delta}_{j}u\|_{L^{p,1}(\mathbb{R}^{n})}\\
&\leq C\sum_{j\leq k}2^{j\sigma}\|\dot{\Delta}_{j}u
\|^{\alpha}_{L^{q,\infty}(\mathbb{R}^{n})}
\|\dot{\Delta}_{j}u\|^{1-\alpha}_{L^{r,\infty}(\mathbb{R}^{n})} +C\sum_{j> k}2^{-j[s-\sigma-n(\f1r-\f1p)]}2^{js}\|\dot{\Delta}_{j}u
\|_{L^{r,\infty}(\mathbb{R}^{n})}\\
&\leq C\f{ 2^{ k[\sigma-s(1-\alpha) ]}}{1-2^{-[\sigma-s(1-\alpha) ]}} \| u\|_{\dot{B}^{0}_{q,\infty,\infty}}^{\alpha} \| u\|_{\dot{B}^{s}_{r,\infty,\infty}}^{1-\alpha}+C\f{ 2^{-k[s-\sigma-n(\f1r-\f1p)]}}{1-2^{- [s-\sigma-n(\f1r-\f1p)]}}  \| u\|_{\dot{B}^{s}_{r,\infty,\infty} }.
\ea$$
Here we need the fact that $\sigma-s(1-\alpha)>0$ with $1/p=\alpha/q+(1-\alpha)/r$, which is derived from the hypothesis that $\f{ 1}{p}
 < (1-\f{\sigma}{s})\f{1}{q}+\f{\sigma}{s}\f{1}{r}$.
Therefore, we obtain the desired estimate
 $$\ba
\|u\|_{\dot{B}^{\sigma}_{p,1,1}} &\leq C \| u\|_{\dot{B}^{0}_{q,\infty,\infty}} ^{\theta} \| u\|_{\dot{B}^{s}_{r,\infty,\infty}}^{1-\theta}.
\ea$$

($III_{22}$) Now, it remains to show \eqref{glgn} with $r<q $ and $\f{ 1}{p}
 > (1-\f{\sigma}{s})\f{1}{q}+\f{\sigma}{s}\f{1}{r}$,
 which imply that
 \be\label{5.5}
 p<q,~~s(1-\alpha)-\sigma >0~~\text{and}~~\f1p=\f{\alpha}{q}+\f{1-\alpha}{r}.
 \ee
 It follows from \eqref{wwkey} and $\f{ 1}{p}
 > (1-\f{\sigma}{s})\f{1}{q}+\f{\sigma}{s}\f{1}{r}$ that
$$
\f{1}{p}<\f1r-\f{s-\sigma}{n}<\f{1}{r},
$$
which together with $p<q$ enables us to derive that
\be\label{5.6}
\sigma-s+n(\f1r-\f1p)>0~~\text{and}~~0<\alpha<1.
\ee
Making use of the Bernstein inequality \eqref{bern2} for Lorentz spaces, \eqref{Interpolation characteristic}, \eqref{5.5} and \eqref{5.6}, we arrive at
 $$\ba
\|u\|_{\dot{B}^{\sigma}_{p,1,1}}&=\sum_{j\leq k}2^{j\sigma}\|\dot{\Delta}_{j}u\|_{L^{p,1}(\mathbb{R}^{n})}+\sum_{j> k}2^{j\sigma}\|\dot{\Delta}_{j}u\|_{L^{p,1}(\mathbb{R}^{n})}\\
&\leq C\sum_{j\leq k} 2^{ j[\sigma-s+n(\f1r-\f1p)]}2^{js}\|\dot{\Delta}_{j}u\|_{L^{r,\infty}(\mathbb{R}^{n})}+
C\sum_{j> k}2^{j\sigma}\|\dot{\Delta}_{j}u
\|^{\alpha}_{L^{q,\infty}(\mathbb{R}^{n})}\|\dot{\Delta}_{j}u
\|^{1-\alpha}_{L^{r,\infty}(\mathbb{R}^{n})}
\\
&\leq C\f{ 2^{ k[   \sigma-s+n(\f1r-\f1p)]}}{1- 2^{- [\sigma -s+n(\f1r-\f1p)]}}  \| u\|_{\dot{B}^{s}_{r,\infty,\infty}} +C\f{ 2^{-k[s(1-\alpha)-\sigma ]}}{1-2^{- [s(1-\alpha)-\sigma ]}} \| u\|_{\dot{B}^{0}_{q,\infty,\infty}}^{\alpha} \| u\|_{\dot{B}^{s}_{r,\infty,\infty}}^{1-\alpha}.
\ea$$
We thereby deduce the inequality
 $$\ba
\|u\|_{\dot{B}^{\sigma}_{p,1,1}} &\leq C  \| u\|_{\dot{B}^{0}_{q,\infty,\infty}} ^{\theta} \| u\|_{\dot{B}^{s}_{r,\infty,\infty}}^{1-\theta}.
\ea$$
The proof of this theorem is completed.
\end{proof}

 \section{Proof of Theorem \ref{the1.4}}
This section is concerned with the application of Gagliardo-Nirenberg inequalities in Lorentz type spaces to the energy conservation
of 3D Navier-Stokes equations. We shall follow the path of \cite{[CL]} to prove \eqref{EIL1}.
\begin{proof}[Proof of Theorem \ref{the1.4}]
As in \cite{[CCFS],[CL]}, since Leray-Hopf weak solutions $v$ satisfy \eqref{NS} in the sense of distributions, there holds for any $Q \in\mathbb{Z},$
$$
\f12\|S_{Q}v(T)\|_{L^{2}(\mathbb{R}^{3})}^{2}+\int_{0}^{T}\|\nabla S_{Q}v\|_{L^{2}(\mathbb{R}^{3})}^{2}ds =\f12\|S_{Q}v_{0}\|_{L^{2}(\mathbb{R}^{3})}^{2}+\int_{0}^{T}\int\text{Tr}(S_{Q}(v\otimes v)\cdot\nabla S_{Q}v)dxds.
$$
In order to get the energy equality, it is enough to show $\int_{0}^{T}\int\text{Tr}(S_{Q}(v\otimes v)\cdot\nabla S_{Q}v)dxds\rightarrow0$ as $Q\rightarrow\infty$. To this end, we recall the following estimates proved in \cite{[CCFS],[CL]}
\be\ba\label{CCFSCL}
&\int_{0}^{T}\int|\text{Tr}(S_{Q}(v\otimes v)\cdot\nabla S_{Q}v)|dxds \\
\leq & C\int_{0}^{T} \left[\sum_{k<Q} 2^{\frac{2k}{3}}\left\|\Delta_{k}v\right\|_{L^{3}(\mathbb{R}^{3})}^{2} 2^{-\frac{4|r-Q|}{3}}\right]^{\frac{3}{2}} ds+C\int_{0}^{T} \left[\sum_{k \geq Q} 2^{\frac{2k}{3}}\left\|\Delta_{k}v\right\|_{L^{3}(\mathbb{R}^{3})}^{2} 2^{-\frac{2|k-Q|}{3}}\right]^{\frac{3}{2}} ds\\
\leq & C\int^{T}_{0}\sum_{k}2^{k}\|\Delta_{k}v\|^{3}_{L^{3}(\mathbb{R}^{3})}
2^{-\f{2|k-Q|}{3}}ds.
\ea\ee

(1)
In view of the interpolation inequality \eqref{Interpolation characteristic} and Young inequality \eqref{YoungGen 2} for Lorentz spaces, we know that
\be\ba\label{5181}
\|\Delta_{k}v\|_{L^{3}(\mathbb{R}^{3})}\leq& C \|\Delta_{k}v\|^{\f{1}{3}}_{L^{2}(\mathbb{R}^{3})}
\|\Delta_{k}v\|^{\f{2}{3}}_{L^{4,\infty}(\mathbb{R}^{3})}\\
\leq& C \|\Delta_{k}v\|^{\f{1}{3}}_{L^{2}(\mathbb{R}^{3})}
\| v\|^{\f{2}{3}}_{L^{4,\infty}(\mathbb{R}^{3})}.
\ea\ee
Then we may derive from  \eqref{5181} that
\be\ba\label{215.2451last}
&\int_{0}^{T}\sum_{k}2^{k}\|\Delta_{k}v\|^{3}_{L^{3}(\mathbb{R}^{3})}
2^{-\f{2|k-Q|}{3}}ds\\\leq &
C\int_{0}^{T}\sum_{k}2^{-\f{2|k-Q|}{3}}
 2^{k}\|\Delta_{k}v\|_{L^{2}(\mathbb{R}^{3})}
\| v\|^{2}_{L^{4,\infty}(\mathbb{R}^{3})}ds \\
\leq &
C\sum_{k}2^{-\f{2|k-Q|}{3}}\B(\int_{0}^{T}2^{ 2k}
 \|\Delta_{k}v\|^{2}_{L^{2}(\mathbb{R}^{3})}ds\B)^{\f{1 }{2 }} \B(\int_{0}^{T}
\| v\|^{4}_{L^{4,\infty}(\mathbb{R}^{3})}ds\B)^{\f{1}{2 }}.
\ea\ee
Next, we show   $\sum_{k}2^{-\f{2|k-Q|}{3}}\B(\int_{0}^{T}2^{ 2 k  }
 \|\Delta_{k}v\|^{2}_{L^{2}(\mathbb{R}^{3})}ds\B)^{\f{1}{2}}\rightarrow 0$ as $Q\rightarrow\infty$. Indeed, since the energy inequality for Leray-Hopf weak solutions guarantees that $v\in L^{\infty}(0,T;L^2(\mathbb{R}^{3}))$ and $\nabla v\in L^{2}(0,T;L^2(\mathbb{R}^{3})),$
 it follows from Parseval's identity of the Fourier transform and the fact that $\varphi\in C^{\infty}(\mathbb{R}^{3})$ is supported in the shell $\{\xi\in\mathbb{R}^{3}: \frac{3}{4}\leq |\xi|\leq \frac{8}{3} \}$ that
\begin{align*}
\sum_{k=0}^{\infty} \int_{0}^{T}2^{ 2 k  }
 \|\Delta_{k}v\|^{2}_{L^{2}(\mathbb{R}^{3})}ds=&\sum_{k=0}^{\infty} \int_{0}^{T}
 \B\|\B[\B|\frac{\cdot}{2^{k}}\B|^{-1}\varphi\B(\frac{\cdot}{2^{k}}\B)\B]\mathcal{F}\B(|\nabla v|\B)\B\|^{2}_{L^{2}(\mathbb{R}^{3})}ds\\
 \leq & \,\frac{16}{9}\sum_{k=0}^{\infty} \int_{0}^{T}
 \B\|\varphi\B(\frac{\cdot}{2^{k}}\B)\mathcal{F}\B(|\nabla v|\B)\B\|^{2}_{L^{2}(\mathbb{R}^{3})}ds\\
 \leq & \,\frac{16}{9}\int_{0}^{T}\int_{\mathbb{R}^{3}}\B(\sum_{k=0}^{\infty}\varphi\B(2^{-k}\xi\B)\B)\B|\mathcal{F}\B(|\nabla v|\B)(\xi)\B|^{2}d\xi ds\\
 \leq & \,\frac{16}{9}\int_{0}^{T}
 \B\|\mathcal{F}\B(|\nabla v|\B)\B\|^{2}_{L^{2}(\mathbb{R}^{3})}ds=\frac{16}{9}\int_{0}^{T}
 \|\nabla v\|^{2}_{L^{2}(\mathbb{R}^{3})}ds<\infty,
\end{align*}
which together with the classical H\"older inequality yields that
\be\label{unkown121}\ba
 &\sum_{k>\f{Q}{2}}2^{-\f{2|k-Q|}{3}}\B(\int_{0}^{T}2^{ 2 k  }\|\Delta_{k}v\|^{2}_{L^{2}(\mathbb{R}^{3})}ds\B)^{\f{1}{2}}\\
 \leq &\B(\sum_{k>\f{Q}{2}}2^{-\f{4|k-Q|}{3}}\B)^{\f{1}{2}} \B(\sum_{k>\f{Q}{2}}\int_{0}^{T}2^{ 2 k  }
 \|\Delta_{k}v\|^{2}_{L^{2}(\mathbb{R}^{3})}ds\B)^{\f{1}{2}}\\
 \leq &\B(2\sum_{k=Q}^{\infty}2^{-\f{4|k-Q|}{3}}\B)^{\f{1}{2}} \B(\sum_{k>\f{Q}{2}}\int_{0}^{T}2^{ 2 k  }
 \|\Delta_{k}v\|^{2}_{L^{2}(\mathbb{R}^{3})}ds\B)^{\f{1}{2}}\\
 \leq & \,2\B(\sum_{k>\f{Q}{2}}\int_{0}^{T}2^{ 2 k  }
 \|\Delta_{k}v\|^{2}_{L^{2}(\mathbb{R}^{3})}ds\B)^{\f{1}{2}}\rightarrow 0,~~\text{as}~~Q\rightarrow\infty.
\ea\ee
On the other hand, we observe that
\begin{align}\label{unkown122}
 &\sum_{k\leq \f{Q}{2}}2^{-\f{2|k-Q|}{3}}\B(\int_{0}^{T}2^{ 2 k  }\|\Delta_{k}v\|^{2}_{L^{2}(\mathbb{R}^{3})}ds\B)^{\f{1}{2}}\notag\\
 \leq &\B(\f12\B)^{\f53+\f{2Q}{3}}\B(\int_{0}^{T}
 \|\Delta_{-1}v\|^{2}_{L^{2}(\mathbb{R}^{3})}ds\B)^{\f{1}{2}}+\f{4}{3}\sum_{0\leq k\leq\f{Q}{2}}2^{-\f{2|k-Q|}{3}}\B(\int_{0}^{T}
 \|\nabla v\|^{2}_{L^{2}(\mathbb{R}^{3})}ds\B)^{\f{1}{2}} \\
 \leq &\B(\f12\B)^{\f53+\f{2Q}{3}}\B(\int_{0}^{T}
 \|v\|^{2}_{L^{2}(\mathbb{R}^{3})}ds\B)^{\f{1}{2}}
 +\B(\f12\B)^{\f{Q}{3}-\f{8}{3}}\B(\int_{0}^{T}
 \|\nabla v\|^{2}_{L^{2}(\mathbb{R}^{3})}ds\B)^{\f{1}{2}}\rightarrow 0,~~\text{as}~~Q\rightarrow\infty.\notag
\end{align}
Here we have used the fact that
$$\int_{0}^{T}\|\Delta_{-1}v\|^{2}_{L^{2}(\mathbb{R}^{3})}ds=\int_{0}^{T}\|\varrho\,\hat{v}\|^{2}_{L^{2}(\mathbb{R}^{3})}ds\leq \int_{0}^{T}\|\hat{v}\|^{2}_{L^{2}(\mathbb{R}^{3})}ds=\int_{0}^{T}\|v\|^{2}_{L^{2}(\mathbb{R}^{3})}ds<\infty.$$
Combining \eqref{unkown121} and \eqref{unkown122}, we thereby conclude that
\be\label{unkown1}\ba
\sum_{k}2^{-\f{2|k-Q|}{3}}\B(\int_{0}^{T}2^{ 2 k  }
 \|\Delta_{k}v\|^{2}_{L^{2}(\mathbb{R}^{3})}ds\B)^{\f{1}{2}}\rightarrow 0,~~\text{as}~~Q\rightarrow\infty.
 \ea\ee

(2)
Let $q>4$. Before going further, we write
$$
I_{Q}=:\{s\in[0,T]:\|v(s)\|_{L^{q,\infty}(\mathbb{R}^{3})}> 2^{\f{Q(q-2)}{q}}\}
$$
 and $I^{c}_{Q}=:[0,T]\setminus I_{Q}$.

According to the interpolation inequality \eqref{Interpolation characteristic} and \eqref{YoungGen 2}, we infer that
\be\ba
\|\Delta_{k}v\|_{L^{3}(\mathbb{R}^{3})}\leq& C \|\Delta_{k}v\|^{\f{2q-6}{3(q-2)}}_{L^{2}(\mathbb{R}^{3})}
\|\Delta_{k}v\|^{\f{q}{3(q-2)}}_{L^{q,\infty}(\mathbb{R}^{3})}\\
\leq& C \|\Delta_{k}v\|^{\f{2q-6}{3(q-2)}}_{L^{2}(\mathbb{R}^{3})}
\| v\|^{\f{q}{3(q-2)}}_{L^{q,\infty}(\mathbb{R}^{3})}.
\ea\ee
Plugging this into \eqref{CCFSCL} and using the H\"older inequality, we observe that
\be\ba\label{215.245}
&\int_{I_{Q}}\sum_{k}2^{k}\|\Delta_{k}v\|^{3}_{L^{3}(\mathbb{R}^{3})}
2^{-\f{2|k-Q|}{3}}ds\\\leq &
C\int_{I_{Q}}\sum_{k}2^{k}2^{-\f{2|k-Q|}{3}}
 \|\Delta_{k}v\|^{\f{2q-6}{ q-2}}_{L^{2}(\mathbb{R}^{3})}
\| v\|^{\f{q}{ q-2}}_{L^{q,\infty}(\mathbb{R}^{3})}ds\\
\leq &
C\sum_{k}\sup_{t\in [0,T]} \|\Delta_{k}v\|^{\f{4q-14}{3(q-2)}}_{L^{2}(\mathbb{R}^{3})} \int_{I_{Q}}2^{-\f{2|k-Q|}{3}}2^{\f{2 k}{3}}
 \|\Delta_{k}v\|^{\f{2 }{3 }}_{L^{2}(\mathbb{R}^{3})}
2^{\f{ k}{3}}\| v\|^{\f{q}{ q-2}}_{L^{q,\infty}(\mathbb{R}^{3})}ds\\
\leq &
C\sup_{t\in [0,T]} \| v\|^{\f{4q-14}{3(q-2)}}_{L^{2}(\mathbb{R}^{3})}   \sum_{k}2^{-\f{2|k-Q|}{3}}\B(\int_{I_{Q}}2^{ 2 k  }
 \|\Delta_{k}v\|^{2}_{L^{2}(\mathbb{R}^{3})}ds\B)^{\f13}2^{\f{ k}{3}}\B(\int_{I_{Q}}
\| v\|^{\f{3 q}{2(q-2)}}_{L^{q,\infty}(\mathbb{R}^{3})}ds\B)^{\f23}.
\ea\ee
The  hypothesis \eqref{EIL1} enables us to obtain
\be\label{215.2}
f_{\ast}(\lambda)=|\{s\in[0,T]:\|v(s)\|_{L^{q,\infty}(\mathbb{R}^{3})}>\lambda\}|\leq  C\lambda^{-\f{2q}{q-2}}.
\ee
This yields that
\be\ba\label{5.4844}
2^{\f{ k}{3}}\B(\int_{I_{Q}}
\| v\|^{\f{3q}{2(q-2)}}_{L^{q,\infty}(\mathbb{R}^{3})}ds\B)^{\f23}=&2^{\f{ k}{3}}\B(\f{3q}{2q-4}\B)^{\f23}\B(\int^{\infty}_{2^{\f{Q(q-2)}{q}}}
\lambda^{\f{3q}{2(q-2)}-1} f_{\ast}(\lambda)  d\lambda+\int^{2^{\f{Q(q-2)}{q}}}_{0}
\lambda^{\f{3q}{2(q-2)}-1} |I_{Q}| d\lambda\B)^{\f23}\\
\leq&C2^{\f{ k}{3}}\B( \int^{\infty}_{2^{\f{Q(q-2)}{q}}}
\lambda^{\f{3q}{2(q-2)}-1-\f{2q}{q-2}}  d\lambda+2^{\f{3Q}{2}}f_{\ast}(2^{\f{Q(q-2)}{q}})\B)^{\f23}\\
\leq&C2^{\f{ k}{3}}\B( \int^{\infty}_{2^{\f{Q(q-2)}{q}}}
\lambda^{-\f{q}{2(q-2)}-1}  d\lambda+2^{-\f{Q}{2}}\B)^{\f23}\\
\leq &C 2^{\f{ k-Q}{3}}.
\ea\ee
Inserting the latter inequality into \eqref{215.245}, we get
\be\ba
&\int_{I_{Q}}\sum_{k}2^{k}\|\Delta_{k}v\|^{3}_{L^{3}(\mathbb{R}^{3})}
2^{-\f{2|k-Q|}{3}}ds\\ \leq &
C(\| v_{0}\| _{L^{2}(\mathbb{R}^{3})} ) \sum_{k}2^{-\f{|k-Q|}{3}}\B(\int_{0}^{T}2^{ 2 k  }
 \|\Delta_{k}v\|^{2}_{L^{2}(\mathbb{R}^{3})}ds\B)^{\f{1}{3}}.
\ea\ee
In the same manner as derivation of \eqref{unkown1}, we arrive at
$$\ba
\sum_{k}2^{-\f{|k-Q|}{3}}\B(\int_{0}^{T}2^{ 2 k  }
 \|\Delta_{k}v\|^{2}_{L^{2}(\mathbb{R}^{3})}ds\B)^{\f{1}{3}}\rightarrow 0,~~\text{as}~~Q\rightarrow\infty.
 \ea$$
Now, it suffices to show  $\int_{I^{c}_{Q}}\sum_{k}2^{k}\|\Delta_{k}v\|^{3}_{L^{3}(\mathbb{R}^{3})}
2^{-\f{2|k-Q|}{3}}ds\rightarrow0\,$ as $\,Q\rightarrow\infty$.\\
A slight modification of deduction of \eqref{215.245} ensures that for any $\varepsilon\in (0,\,q-4],$
 \be\ba\label{215.2451}
&\int_{I^{c}_{Q}}\sum_{k}2^{k}\|\Delta_{k}v\|^{3}_{L^{3}(\mathbb{R}^{3})}
2^{-\f{2|k-Q|}{3}}ds\\\leq &
C\int_{I^{c}_{Q}}\sum_{k}2^{k}2^{-\f{2|k-Q|}{3}}
 \|\Delta_{k}v\|^{\f{2q-6}{ q-2 }}_{L^{2}(\mathbb{R}^{3})}
\| v\|^{\f{q}{ q-2 }}_{L^{q,\infty}(\mathbb{R}^{3})}ds\\
\leq &
C\sum_{k} \sup_{t\in [0,T]} \|\Delta_{k}v\|^{\f{ 2q-8-2\varepsilon }{(2+\varepsilon)(q-2)}}_{L^{2}(\mathbb{R}^{3})}   \int_{I^{c}_{Q}}2^{-\f{2|k-Q|}{3}}2^{\f{2(1+\varepsilon) k}{2+\varepsilon}}
 \|\Delta_{k}v\|^{\f{2(1+\varepsilon) }{2+\varepsilon}}_{L^{2}(\mathbb{R}^{3})}
2^{ -\f{\varepsilon k}{2+\varepsilon}}\| v\|^{\f{q}{ q-2 }}_{L^{q,\infty}(\mathbb{R}^{3})}ds\\
\leq &
C\sup_{t\in [0,T]} \| v\|^{\f{ 2q-8-2\varepsilon }{(2+\varepsilon)(q-2)}}_{L^{2}(\mathbb{R}^{3})}   \sum_{k}2^{-\f{2|k-Q|}{3}}\B(\int_{I_{Q}^{c}}2^{ 2k}
 \|\Delta_{k}v\|^{2}_{L^{2}(\mathbb{R}^{3})}ds\B)^{\f{1+\varepsilon}{2+\varepsilon}}2^{ -\f{\varepsilon k}{2+\varepsilon}}\B(\int_{I_{Q}^{c}}
\| v\|^{\f{q(2+\varepsilon)} {q-2}}_{L^{q,\infty}(\mathbb{R}^{3})}ds\B)^{\f{1}{2+\varepsilon}}.
\ea\ee
Then it follows from \eqref{215.2} that
$$\ba
2^{ -\f{\varepsilon k}{2+\varepsilon}}\B(\int_{I_{Q}^{c}}
\| v\|^{\f{q(2+\varepsilon)}{q-2}}_{L^{q,\infty}(\mathbb{R}^{3})}ds\B)^{\f{1}{2+\varepsilon}}
\leq &2^{ -\f{\varepsilon k}{2+\varepsilon}}\B(\f{q(2+\varepsilon)}{q-2}\int_{0}^{2^{\f{Q(q-2)}{q}}}
\lambda^{\f{q(2+\varepsilon)}{q-2}-1} f_{\ast}(\lambda)  d\lambda\B)^{\f{1}{2+\varepsilon}}\\
\leq&C2^{ -\f{\varepsilon k}{2+\varepsilon}}\B( \int_{0}^{2^{\f{Q(q-2)}{q}}}
\lambda^{\f{q(2+\varepsilon)}{q-2}-1-\f{2q}{q-2}}  d\lambda\B)^{\f{1}{2+\varepsilon}}\\
\leq&C2^{ -\f{\varepsilon k}{2+\varepsilon}}\B( \int_{0}^{2^{\f{Q(q-2)}{q}}}
\lambda^{\f{q \varepsilon}{q-2}-1}  d\lambda\B)^{\f{1}{2+\varepsilon}}\\
\leq &C 2^{\f{ \varepsilon(Q-k)}{2+\varepsilon}}.
\ea$$
Inserting the latter inequality into \eqref{215.2451}, we get
\be\ba
&\int_{I_{Q}^{c}}\sum_{k}2^{k}\|\Delta_{k}v\|^{3}_{L^{3}(\mathbb{R}^{3})}
2^{-\f{2|k-Q|}{3}}ds\\ \leq &
C(\| v_{0}\| _{L^{2}(\mathbb{R}^{3})} ) \sum_{k}2^{-\f{(4-\varepsilon)|k-Q|}{3(2+\varepsilon)}}\B(\int_{0}^{T}2^{ 2 k  }
 \|\Delta_{k}v\|^{2}_{L^{2}(\mathbb{R}^{3})}ds\B)^{\f{1+\varepsilon}{2+\varepsilon}}.
\ea\ee
Take the positive constant $\varepsilon<\min\{4,\,q-4\}$. By a similar argument like \eqref{unkown1}, we also have $\int_{I_{Q}^{c}}\sum_{k}2^{k}\|\Delta_{k}v\|^{3}_{L^{3}(\mathbb{R}^{3})}
2^{-\f{2|k-Q|}{3}}ds\rightarrow0\,$ as $\,Q\rightarrow\infty$.

(3)
In light of the Gagliardo-Nirenberg inequality \eqref{GNL} for Lorentz spaces and \eqref{lincreases}, we obtain
$$
\|v\|_{L^{4}(\mathbb{R}^{3})} \leq C\|\nabla v\|_{L^{2}(\mathbb{R}^{3})}^{\frac{3(4-q)}{2(6-q)}}\|v\|_{L^{q,\infty}(\mathbb{R}^{3})}^{\frac{q}{12-2q}}, ~~\text { with}  ~~3< q < 4.
$$
Hence, the H\"older inequality entails that
$$\ba
\int_{0}^{T}\|v\|_{L^4(\mathbb{R}^{3})}^{4} d s \leq&C \int_{0}^{T}\|\nabla v\|_{L^{2}(\mathbb{R}^{3})}^{\frac{6(4-q)}{6-q}}\|v\|_{L^{q,\infty}(\mathbb{R}^{3})}^{\frac{2q}{6-q}}
  d s\\
\leq & C\B(\int_{0}^{T}\|\nabla v\|^{2}_{L^{ 2}(\mathbb{R}^{3})}d s\B)^{\frac{3(4-q)}{6-q}}\B(\int_{0}^{T}
\|v\|^{\f{q}{q-3}}_{L^{q,\infty}(\mathbb{R}^{3})}
  d s\B)^{\frac{2(q-3)}{6-q}}.
\ea$$
This means that $v\in L^{4}(0,T;L^{4}(\mathbb{R}^{3}) ),$ which helps us to get
energy conservation.

(4)
Since $3/2<q<9/5$, we may choose an index $\tilde{q}\in (q,\,9/5)$ such that $\,1/\tilde{q}=(1-\alpha)/2+\alpha/q\,$ with some $\alpha\in(0,\,1).$ Let $1/\tilde{p}+3/\tilde{q}=2.$ Then it follows from $q>3/2$ that $1/\tilde{p}=2-3(1-\alpha)/2-3\alpha/q>(1-\alpha)/2.$

Thanks to the interpolation characteristic \eqref{Interpolation characteristic} of Lorentz spaces, \eqref{lincreases} and the H\"older inequality, we arrive at
 $$\ba
\int_{0}^{T}\|\nabla v\|^{\tilde{p}}_{L^{\tilde{q}}(\mathbb{R}^{3})} d s \leq& C\int_{0}^{T}\|\nabla v\|^{(1-\alpha)\tilde{p}}_{L^{2}(\mathbb{R}^{3})}
\|\nabla v\|^{\alpha\tilde{p}}_{L^{q,\infty}(\mathbb{R}^{3})} d s \\
\leq & C\B(\int_{0}^{T}\|\nabla v\|^{2}_{L^{ 2}(\mathbb{R}^{3})}d s\B)^{\frac{(1-\alpha)\tilde{p}}{2}}\B(\int_{0}^{T}
\|\nabla v\|^{p}_{L^{q,\infty}(\mathbb{R}^{3})} d s\B)^{1-\frac{(1-\alpha)\tilde{p}}{2}},
\ea$$
where we have used the fact that
$$  \frac{2\alpha \tilde{p}}{2-(1-\alpha)\tilde{p}}=p\,, ~~\text{and}~~0<\frac{(1-\alpha)\tilde{p}}{2}<1.$$
Hence, we conclude the desired energy equality from the known result \eqref{bcz}.

(5)
Note that $1/s<p<3,$ we may choose an index $p_{1}$ such that $\max\{1,p\}<p_{1}<\min\{3,sp\}.$ Take $\theta=1-p/p_{1},$ then $0<\theta<1-1/s.$ Let $1/p_{1}+6/(5q_{1})=1,$ which implies that $\,9/5 <q_{1}<\infty\,$ and
$$  \frac{3}{q_{1}}-1= \frac{3\theta}{2}+(1-\theta)\left(\frac{3}{2}-\frac{5}{2p}\right)= \frac{3\theta}{2}+(1-\theta)\left(\frac{3}{q}-s\right).  $$
By the Gagliardo-Nirenberg inequality  \eqref{glgn} for Besov-Lorentz spaces and \eqref{lincreases}, we find that
$$\ba
\int_{0}^{T}\|\nabla v\|^{p_{1}}_{L^{q_{1}}(\mathbb{R}^{3})}ds\leq& C\int_{0}^{T} \|  v\|^{\theta p_{1}}_{L^{2}(\mathbb{R}^{3})} \|v\|^{(1-\theta)p_{1}}_{\dot{B}^{s}_{q,\infty,\infty}}ds \\
\leq& C\|v\|^{p_{1}-p }_{L^{\infty}(0,T;L^2(\mathbb{R}^{3}))}\int_{0}^{T}\|v\|^{p}_{\dot{B}^{s}_{q,\infty,\infty}}ds.
\ea$$
This together with the  known result \eqref{bcz} yields the energy equality.

Consequently, we complete the proof of this theorem.
  \end{proof}

\section*{Acknowledgements}
%The authors thank the anonymous referee and the associated editor for the invaluable
%comments and suggestions which helped to improve the paper greatly.
The authors would like to express their sincere gratitude to Dr. Xiaoxin Zheng at Beihang University,
for a lot of useful discussion on this topic.
Wang was partially supported by  the National Natural
Science Foundation of China under grant (No. 11971446, No. 12071113   and  No. 11601492).
Wei was partially supported by the National Natural Science Foundation of China under grant (No. 11601423, No. 11771352, No. 11871057).
Ye was partially supported by the National Natural Science Foundation of China  under grant (No. 11701145) and China Postdoctoral Science Foundation (No. 2020M672196).

%\end{CJK*}
\end{document}